\documentclass[10pt]{amsart}

\usepackage{amsthm,amsfonts,amsmath,verbatim,amssymb,verbatim,cite}

\usepackage{hyperref}

\usepackage{mathrsfs}
\usepackage[usenames]{color}
\usepackage{tikz,graphicx}

\numberwithin{equation}{section}
\newtheorem{theorem}{Theorem}[section]
\newtheorem{proposition}[theorem]{Proposition}
\newtheorem{lemma}[theorem]{Lemma}
\newtheorem{claim}[theorem]{Claim}

\DeclareMathOperator{\CT}{CT}
\DeclareMathOperator{\eup}{e}
\DeclareMathOperator{\id}{id}

\newcommand{\blue}[1]{{\color{blue}  #1}}
\newcommand{\Part}{\mathscr P}
\newcommand{\Comp}{\mathscr C}
\newcommand{\Mat}{\mathscr M}
\newcommand{\Rat}{\mathbb Q}
\newcommand{\Real}{\mathbb R}
\newcommand{\Z}{\mathbb Z}
\newcommand{\F}{\mathbb F}
\newcommand{\Symm}{\mathfrak{S}}
\newcommand{\abs}[1]{\lvert#1\rvert}
\newcommand{\la}{\lambda}
\newcommand{\bil}[2]{(#1,#2)}
\newcommand{\mult}{\mathfrak m}
\newcommand{\ip}[2]{\langle#1,#2\rangle}
\newcommand{\qbin}[2]{\genfrac{[}{]}{0pt}{}{#1}{#2}}

\newcommand{\bc}{c}
\newcommand{\bt}{\boldsymbol{t}}
\newcommand{\bs}{\boldsymbol{s}}
\newcommand{\btau}{\boldsymbol{\tau}}
\newcommand{\A}{\sigma}
\newcommand{\An}{\abs{a}}
\newcommand{\coef}{\mathsf{c}}

\begin{document}

\title{Constant term identities and Poincar\'e polynomials}

\author{Gyula K\'arolyi}
\address{School of Mathematics and Physics,
The University of Queensland, Brisbane, QLD 4072, Australia
and Institute of Mathematics, E\"otv\"os University, 
Budapest, Hungary}

\author{Alain Lascoux}
\address{CNRS, Institut Gaspard Monge,
Universit\'e Paris-Est, Marne-la-Vall\'ee, France}

\author{S. Ole Warnaar}
\address{School of Mathematics and Physics,
The University of Queensland, Brisbane, QLD 4072, Australia}

\thanks{Work supported by the Australian Research Council
and by Hungarian National Scientific Research Funds (OTKA) Grant K100291.}


\begin{abstract}
In 1982 Macdonald published his now famous constant term conjectures for
classical root systems.
This paper begins with the almost trivial observation that Macdonald's
constant term identities admit an extra set of free parameters, thereby 
linking them to Poincar\'e polynomials.
We then exploit these extra degrees of freedom in the case of type 
$\mathrm{A}$ to give the first proof of Kadell's orthogonality
conjecture---a symmetric function generalisation of the $q$-Dyson 
conjecture or Zeilberger--Bressoud theorem.

\noindent
Key ingredients in our proof of Kadell's orthogonality conjecture are 
multivariable Lagrange interpolation, the scalar product for
Demazure characters and $(0,1)$-matrices.

\end{abstract}

\maketitle

\section{Introduction and summary of results}

Given a finite real reflection group $W$, the classical Poincar\'e
polynomial $W(t)$ is defined as
\cite{Bourbaki02,Humphreys90}
\begin{equation}\label{Eq_Poincare-polynomials}
W(t)=\sum_{w\in W} t^{l(w)},
\end{equation}
where $l$ is the length function on $W$.
A key result in the theory of reflection groups is the 
Chevalley--Solomon product formula \cite{Chevalley55,Solomon66}
\begin{equation}\label{Eq_Wtd}
W(t)=\prod_{i=1}^r \frac{1-t^{d_i}}{1-t},
\end{equation}
where the $d_1,\dots,d_r\geq 2$ are the degrees of the fundamental 
invariants.

For reflection groups of crystallographic type, i.e., Weyl groups,
Macdonald \cite{Macdonald72} generalised the Poincar\'e polynomial
to a multivariable polynomial $W(\bt)$ by attaching a variable
$t_{\alpha}$ to each positive root $\alpha$ of the underlying root 
system.
To be more precise, let $R$ be a reduced irreducible root system 
of rank $r$, and $R^{+}$ (resp. $R^{-}=-R^{+}$) the set of positive
(resp. negative) roots. 
Write $\alpha>0$ if $\alpha\in R^{+}$.
Let $\bt$ denote the alphabet $\bt=\{t_{\alpha}\}_{\alpha>0}$ and for
$S\subseteq R^{+}$ define the word $t_S=\prod_{\alpha\in S}t_{\alpha}$.
Macdonald's multivariable Poincar\'e polynomial is then given by
\[
W(\bt)=\sum_{w\in W} t_{R(w)},
\]
where $R(w)=R^{+}\cap w(R^{-})$. 
Since $\abs{R(w)}=l(w)$, the multivariable Poincar\'e polynomial reduces 
to \eqref{Eq_Poincare-polynomials} when $t_{\alpha}=t$ for all $\alpha$.

In its full generality $W(\bt)$ no longer admits a product form. 
Instead Macdonald \cite[Theorem 2.8]{Macdonald72} showed
(see also \cite[Theorem 1]{Stembridge98}) that it may be expressed as 
\begin{equation}\label{Eq_Poincare}
W(\bt)=
\sum_{w\in W} \prod_{\alpha>0}
\frac{1-t_{\alpha} \eup^{-w(\alpha)}}{1-\eup^{-w(\alpha)}},
\end{equation}
with $\eup^{\alpha}$ a formal exponential.

\medskip

A different discovery of Macdonald, formulated as a conjecture, is the
constant term identity
\cite{Macdonald82}
\begin{equation}\label{Eq_Macdonald-q-CT}
\CT\bigg[\prod_{\alpha>0} (\eup^{-\alpha},q \eup^{\alpha})_k \bigg] 
=\prod_{i=1}^r \qbin{d_i k}{k},
\end{equation}
where 
\[
(a)_k=\prod_{i=0}^{k-1}(1-aq^i)\quad\text{and}\quad
(a_1,\dots,a_m)_k=(a_1)_k\cdots(a_m)_k
\]
are $q$-shifted factorials and
\[
\qbin{n}{m}=\begin{cases}\displaystyle
\frac{(q)_n}{(q)_m(q)_{n-m}} & \text{for $0\leq m\leq n$,} \\[3mm]
0 & \text{otherwise}
\end{cases}
\]
is a $q$-binomial coefficient \cite{Andrews76,GR04}. 
There is a large literature on \eqref{Eq_Macdonald-q-CT}, see e.g., 
\cite{FW08} and references therein, with an ultimately case-free proof 
found by Cherednik \cite{Cherednik93,Cherednik95} based on his double 
affine Hecke algebra \cite{Cherednik05}.

It is not at all difficult to also express the multivariable Poincar\'e
polynomial as a constant term:
\begin{equation}\label{Eq_Poincare-CT}
W(\bt)=\CT\bigg[\prod_{\alpha>0}
(1-\eup^{-\alpha})(1-t_{\alpha}\eup^{\alpha}) \bigg].
\end{equation}
Since this generalises the $k=1$ case of Macdonald's conjecture, it 
then takes little to note that \eqref{Eq_Macdonald-q-CT} and 
\eqref{Eq_Poincare-CT} can in fact be unified.
\begin{proposition}\label{Prop_CTt}
For $k\geq 1$ we have
\begin{equation}\label{Eq_CTW}
\CT\bigg[\prod_{\alpha>0}
(1-\eup^{-\alpha})(1-t_{\alpha}\eup^{\alpha})
(q\eup^{-\alpha},q\eup^{\alpha})_{k-1}\bigg] =
W(\bt) \prod_{i=1}^r \qbin{d_i k-1}{k-1}.
\end{equation}
\end{proposition}
This is perhaps an elegant result---connecting Poincar\'e polynomials
and Mac\-do\-nald-type constant term identities---but, assuming 
\eqref{Eq_Macdonald-q-CT}, not at all deep.
More interesting is what happens if one restricts considerations to the
root system $\mathrm{A}_{n-1}$ for which the set of positive roots may
be taken to be $R^{+}=\{\epsilon_i-\epsilon_j:~1\leq i<j\leq n\}$,
with $\epsilon_i$ the $i$th standard unit vector in $\Real^n$. 
Then \eqref{Eq_Macdonald-q-CT} admits the following inhomogeneous 
generalisation known as the Andrews' $q$-Dyson conjecture 
\cite{Andrews75} or the Zeilberger--Bressoud theorem \cite{ZB85}
(see also \cite{BG85,GX06,KN12}):
\begin{equation}\label{Eq_q-Dyson}
\CT\bigg[ \prod_{1\leq i<j\leq n}
(x_i/x_j)_{a_i}(qx_j/x_i)_{a_j} \bigg]=
\frac{(q)_{a_1+\cdots+a_n}}{(q)_{a_1}\cdots (q)_{a_n}},
\end{equation}
where we have identified $\exp(-\epsilon_i)$ with $x_i$. 
In view of Proposition~\ref{Prop_CTt} it is natural to try to generalise
\eqref{Eq_q-Dyson} by replacing 
\[
(qx_j/x_i)_{a_j}\quad\text{by}\quad (qx_j/x_i)_{a_j-1}(1-t_{ij}x_j/x_i).
\]
Here we have made the further identification of 
$t_{\epsilon_i-\epsilon_j}$ with $t_{ij}$.
To describe the resulting constant term identity we need some further 
notation.

Let $\binom{[n]}{2}=\{(i,j):~1\leq i<j\leq n\}$ and for $S\subseteq
\binom{[n]}{2}$ set $t_S=\prod_{(i,j)\in S} t_{ij}$.
Let $I(w)=\{(i,j)\in \binom{[n]}{2}:~w(i)>w(j)\}$ be the inversion set
of the permutation $w\in\Symm_n$ and 
$R(w)=I(w^{-1})=\big\{(w(j),w(i))\in \binom{[n]}{2}:~i<j\big\}$.
For $a=(a_1,\dots,a_n)$ a sequence of nonnegative integers we write
$\An$ and $\A_i$ for the sum of its components and the
$i$th partial sum respectively, i.e.,
$\An=a_1+\cdots+a_n$ and $\A_i=a_1+\cdots+a_i$.
The notation $\An$ and $\A_n$ will be used interchangeably.
The $q$-multinomial coefficient $\qbin{\An}{a}$ can now be defined as
\begin{equation}\label{Eq_qmultinomial}
\qbin{\An}{a}=\frac{(q)_{\An}}{(q)_{a_1}\cdots(q)_{a_n}}=
\prod_{i=1}^n \qbin{\A_i}{a_i} .
\end{equation}
Finally we introduce an inhomogeneous version of the multivariable
$\mathrm{A}_{n-1}$ Poincar\'e polynomial, with coefficients in $\Rat(q)$,
as follows:
\[
W_a(\bt)=\sum_{w\in\Symm_n} t_{R(w)}
\prod_{i=1}^n \frac{1-q^{\A_i}}{1-q^{w(\A_i)}} ,
\]
where $w(\A_i)=a_{w(1)}+\cdots+a_{w(i)}$.
Note that $W_{(k,\dots,k)}(\bt)=W(\bt)$.

\begin{theorem}\label{Thm_Poincare-q-Dyson}
Let $a=(a_1,\dots,a_n)$ be a sequence of positive integers. Then
\begin{equation}\label{Eq_Poincare-q-Dyson}
\CT\bigg[ \prod_{1\leq i<j\leq n}
(x_i/x_j)_{a_i}(qx_j/x_i)_{a_j-1}(1-t_{ij} x_j/x_i)\bigg]
=W_a(\bt)\prod_{i=1}^n \qbin{\sigma_i-1}{a_i-1} .
\end{equation}
\end{theorem}
For $a_1=\dots=a_n=k$ the right-hand side simplifies to
\[
W(\bt)\prod_{i=1}^{n-1} \qbin{(i+1)k-1}{k-1},
\]
and we recover the $\mathrm{A}_{n-1}$ case of \eqref{Eq_CTW}.

Our proof of Theorem~\ref{Thm_Poincare-q-Dyson} uses the polynomial 
lemma of Laso\'n \cite{Lason10} and Karasev--Petrov \cite{KP12}---a 
form of multivariable 
Lagrange interpolation in the spirit of the Combinatorial 
Nullstellensatz \cite{Alon99}.
The efficacy of the polynomial lemma to constant term identities was
recently demonstrated in \cite{KN12} in the form of a one-page proof of
the $q$-Dyson conjecture \eqref{Eq_q-Dyson}.

Several constant term identities due to Bressoud and Goulden \cite{BG85}
follow from Theorem~\ref{Thm_Poincare-q-Dyson} in a very easy manner.
If $t_{ij}=0$ for all $i$ and $j$ then $W_a(\bt)=1$ and we obtain
\cite[Theorem 2.2, $\sigma=\id$]{BG85}.
More generally, for $I\subseteq \{1,2,\dots,n\}$
and $\bar{I}$ its complement, let $t_{ij}=0$ if $j\in I$ and 
$t_{ij}=q^{a_j}$ if $j\in\bar{I}$.
Then only those permutations $w$ contribute to $W_a(\bt)$ for which
$w(i)=i$ for $i\in I$. Replacing the sequence $(q^{a_j})_{j\in\bar{I}}$
by $(u_1,\dots,u_m)$ ($m:=\abs{\bar{I}}$), we are left with the
simple computation (see Page \pageref{pagelabel})
\begin{equation}\label{Eq_usum}
\sum_{w\in\Symm_m}w\bigg(\,\prod_{i=1}^m\frac{1-u_i}
{1-u_1\cdots u_i}\bigg) \prod_{(i,j)\in R(w)}u_j=1.
\end{equation}
Hence $W_a(\bt)=\prod_{i\in\bar{I}} (1-q^{\A_i})/(1-q^{a_i})$, and
we obtain the generalised $q$-Dyson identity 
\cite[Theorem 2.5]{BG85}
\[
\CT\bigg[\prod_{1\leq i<j\leq n}
(x_i/x_j)_{a_i}(qx_j/x_i)_{a_j-\chi(j\in I)} \bigg]=
\qbin{\An}{a} \prod_{i\in I} \frac{1-q^{a_i}}{1-q^{\A_i}},
\]
where $\chi$ is the indicator function.
Similarly, if $t_{ij}=-1$ we may use
\[
\sum_{w\in\Symm_n} (-1)^{l(w)} 
w\bigg(\,\prod_{i=1}^n \frac{1-u_i}{1-u_1\cdots u_i}\bigg)=
\prod_{1\leq i<j\leq n} \frac{u_i-u_j}{1-u_iu_j}
\]
to find $W_a(\bt)=\prod_{i=1}^n (1-q^{\A_i})/(1-q^{a_i})
\prod_{i<j}(q^{a_i}-q^{a_j})/(1-q^{a_i+a_j})$. This results in
\cite[Theorem 2.7]{BG85}
\[
\CT\bigg[\prod_{1\leq i<j\leq n} (x_j/x_i-x_i/x_j)
\prod_{i\neq j}(qx_i/x_j)_{a_i-1} \bigg]=
\qbin{\An}{a}\prod_{1\leq i<j\leq n} 
\frac{q^{a_i}-q^{a_j}}{1-q^{a_i+a_j}}.
\]
The fact that Theorem~\ref{Thm_Poincare-q-Dyson} allows us to reprove
the constant term identities of Bressoud and Goulden is not too
surprising. Combining two of the key theorems of their paper---both
formulated in the language of tournaments---and reinterpreting these 
in terms of permutations provides an alternative method of proof of 
Theorem~\ref{Thm_Poincare-q-Dyson}. 

\medskip

As a much deeper and more interesting application than the reproof of 
known results, we will show that Theorem~\ref{Thm_Poincare-q-Dyson} 
may be used to prove Kadell's $q$-Dyson orthogonality 
conjecture~\cite{Kadell00}.

For $x=(x_1,\dots,x_n)$ and $u=(u_1,\dots,u_n)$ a sequence of integers,
write $x^u$ for the monomial $x_1^{u_1}\cdots x_n^{u_n}$.
If all $u_i$ are nonnegative we refer to $u$ as a composition and, 
if in addition $u_1+\cdots+u_n=k$, we write $\abs{u}=k$.
The set of all compositions of the form $(u_1,\dots,u_n)$ will be denoted by
$\Comp_n$.
If $u\in\Comp_n$ satisfies $u_1\geq u_2\geq \cdots \geq u_n$
we say that $u$ is a partition, and $\Part_n$ will denote
the set of all partitions in $\Comp_n$.
As is customary, we will often denote partitions by the 
Greek letters $\la,\mu,\nu$ and not display their tails of zeros.
The unique partition in the $\Symm_n$ orbit of $u\in\Comp_n$ 
is denoted by $u^+$.
Let $s_{\la}(x)$ be the classical Schur function \cite{Macdonald95}
\[
s_{\la}(x)=\frac{\det_{1\leq i,j\leq n}\big(x_i^{\la_j+n-j}\big)}
{\prod_{1\leq i<j\leq n} (x_i-x_j)}
\]
for $\la\in\Part_n$.
For $\la=(m)$ this simplifies to the $m$th complete symmetric function
\[
h_m(x)=\sum_{\abs{v}=m} x^v.
\]
Given a sequence of nonnegative integers $a=(a_1,\dots,a_n)$ let
\begin{equation}\label{Eq_xa}
x^{(a)}=\big(x_1,x_1q,\dots,x_1q^{a_1-1},\dots,
x_n,x_nq,\dots,x_nq^{a_n-1}\big)
\end{equation}
or, in the notation of $\la$-rings \cite{Lascoux03},
\[
x^{(a)}=
x_1\frac{1-q^{a_1}}{1-q}+\cdots+x_n \frac{1-q^{a_n}}{1-q}.
\]
We will be interested in constant terms of the form
\begin{equation}\label{Eq_D_vla}
D_{v,\la}(a):=\CT\bigg[x^{-v} s_{\la}\big(x^{(a)}\big) 
\prod_{1\leq i<j\leq n} (x_i/x_j)_{a_i} (qx_j/x_i)_{a_j} \bigg],
\end{equation}
for $v\in\Comp_n$ and $\la\in\Part_{\An}$.
Before stating Kadell's conjecture we make a few general comments about
the above constant term. Firstly, by homogeneity $D_{v,\la}(a)=0$ if 
$\abs{v}\neq\abs{\la}$. Secondly, if $n=1$, and viewing $v$ and $a$ as 
scalars,
\[
D_{v,\la}(a)=s_{\la}\Big(\frac{1-q^a}{1-q}\Big)
\CT\big[x^{\abs{\la}-v}\big]=
\delta_{v,\abs{\la}} 
q^{\sum_{i<j}\lambda_j}\prod_{s\in\la} \frac{1-q^{a+c(s)}}{1-q^{h(s)}},
\]
where $c(s)$ and $h(s)$ are the content
and hook-length of the square $s$ in the diagram of $\la$, see 
\cite[p. 44]{Macdonald95}.
As a third remark we note that if $a_k=0$ then 
$s_{\la}\big(x^{(a)}\big)$ does not depend on $x_k$ and the double
product over $i<j$, viewed as a function of $x_k$, takes the form 
$c_0+c_{-1}x_k^{-1}+c_{-2}x_k^{-2}+\cdots$.
Hence the constant term will be zero unless $v_k=0$.
Lastly, it is natural to more generally consider \eqref{Eq_D_vla} for $v$ an
arbitrary element of $\Z^n$. Unfortunately, the method developed 
for proving Kadell's orthogonality conjecture has little to say 
about this more general range of $v$;
as we shall see later, the $i$th entry of the composition $v$ arises as
the $i$th row sum of a $(0,1)$-matrix. 
Only in Section~\ref{Sec_Questions} will we consider \eqref{Eq_D_vla}
for $v\not\in\Comp_n$.

\begin{theorem}[Kadell's orthogonality conjecture \cite{Kadell00}]
\label{Thm_Kadell-conjecture}
For $m$ a positive integer, $v\in\Comp_n$ and $a=(a_1,\dots,a_n)$ a
sequence of nonnegative integers,
\begin{subequations}
\begin{equation}\label{Eq_Kadell-nul}
D_{v,(m)}(a)=0 \quad \text{if $v^{+}\neq (m)$}
\end{equation}
and
\begin{equation}\label{Eq_Kadell-niet-nul}
D_{v,(m)}(a)=
\frac{q^{\A_n-\A_k} (1-q^{a_k})(q^{\abs{a}})_m}
{(1-q^{\abs{a}})(q^{\An-a_k+1})_m} \qbin{\abs{a}}{a}
\quad \text{if $v=(\underbrace{0,\dots,0}_{k-1\text{ times}},m,
\underbrace{0,\dots,0}_{n-k\text{ times}})$.}
\end{equation}
\end{subequations}
\end{theorem}
We remark that Kadell's original conjecture only includes the $k=1$ case
of \eqref{Eq_Kadell-niet-nul}, i.e., $v=(m)$, and misses both the term 
$q^{\A_n-\A_1}$ and the $+1$ in $(q^{\An-a_1+1})_m$.

We obtain several more general results than 
Theorem~\ref{Thm_Kadell-conjecture} involving Schur functions.
A particularly simple example is
\begin{equation}\label{Eq_strict}
D_{\la,\la}(a)=q^{\sum_{i<j}a_j}\prod_{i=1}^n 
\qbin{\la_i+a_i+\cdots+a_n-1}{a_i-1},
\end{equation}
provided $\la$ is a strict partition, i.e., 
$\la_1>\la_2>\cdots>\la_n\geq 0$ and all $a_i>0$.

\medskip

The remainder of this paper is organised as follows. In the next section
we give a simple proof of an inhomogeneous version of 
Proposition~\ref{Prop_CTt}. Then, in Section~\ref{Sec_PqD}, we use the
polynomial lemma to give a proof of Theorem~\ref{Thm_Poincare-q-Dyson}
and show how the theorem relates to constant term identities of Bressoud 
and Goulden.
In Section~\ref{Sec_Kadell}, we apply
Theorem~\ref{Thm_Poincare-q-Dyson} to prove and generalise Kadell's
orthogonality conjecture.
Finally, in Section~\ref{Sec_Questions}, answering a question raised by
the anonymous referee, we show that Kadell's orthogonality conjecture
implies a conjecture of Sills \cite{Sills06} proved previously by
Lv, Xin and Zhou using different means \cite{LXZ09}.

\section{Proposition~\ref{Prop_CTt} and its inhomogeneous extension}

This section, which is elementary in its contents, may be viewed as a 
warm-up exercise to the more involved considerations of subsequent sections.
We do however assume the reader has a basic knowledge of root systems,
see e.g., \cite{Humphreys78,Bourbaki02}.

Although \eqref{Eq_Poincare-CT} is a special case of 
Proposition~\ref{Prop_CTt}, we establish it prior to proving the more 
general result. 
We make this distinction
because the former requires little more than the Weyl denominator 
formula whereas our proof of Proposition~\ref{Prop_CTt}
relies on the deep result \eqref{Eq_Macdonald-q-CT}.

Let $\rho=(1/2)\sum_{\alpha>0}\alpha$ be the Weyl vector of $R$.
By the Weyl denominator formula
\[
\prod_{\alpha>0} (1-\eup^{-\alpha})=
\sum_{w\in W} (-1)^{l(w)} \eup^{w(\rho)-\rho}
\]
and the expansion
\begin{equation}\label{Eq_expansion}
\prod_{\alpha>0} (1-t_{\alpha}\eup^{\alpha})=
\sum_{S\subseteq R^{+}}(-1)^{\abs{S}}t_S\eup^{\sum_{\alpha\in S}\alpha},
\end{equation}
we have
\[
\CT\bigg[\prod_{\alpha>0}
(1-\eup^{-\alpha})(1-t_{\alpha}\eup^{\alpha}) \bigg] 
=\sum_{w\in W} \sum_{S\subseteq R^{+}}(-1)^{l(w)+\abs{S}} 
t_S \CT\Big[\eup^{w(\rho)-\rho+\sum_{\alpha\in S}\alpha}\Big].
\]
Since $w(\rho)-\rho=-\sum_{\alpha\in R(w)} \alpha$ (see e.g.,
\cite[p. 167]{Macdonald72}) the right-hand side may also be written as
\[
\sum_{w\in W}\sum_{S\subseteq R^{+}}(-1)^{l(w)+\abs{S}}t_S
\CT\Big[\eup^{\sum_{\alpha\in S}
\alpha-\sum_{\alpha\in R(w)}\alpha}\Big].
\]
The constant term vanishes unless $S=R(w)\subseteq R^{+}$ so that we are
left with $\sum_{w\in W} t_{R(w)}=W(\bt)$ as claimed.

Next we turn to the proof of the more general Proposition~\ref{Prop_CTt}.
In fact, what we shall prove is an inhomogeneous version of the 
proposition which generalises another ex-conjecture of Macdonald, also 
proved by Cherednik.
Let $\{k_{\alpha}\}_{\alpha>0}$ be a set of integers constant along Weyl
orbits, i.e., $k_{\alpha}=k_{\beta}$ for $\|\alpha\|=\|\beta\|$. 
Then \cite[Conjecture 2.3]{Macdonald82}, \cite[Theorem 1.1]{Cherednik93},
\cite[Theorem 0.1]{Cherednik95}
\begin{equation}\label{Eq_Macdonald-q-CT-inhom}
\CT\bigg[\prod_{\alpha>0} (\eup^{-\alpha},q \eup^{\alpha})_{k_{\alpha}}
\bigg] 
=\prod_{\alpha>0} 
\frac{(q)_{\bil{\rho_k}{\alpha^{\vee}}+k_{\alpha}}
(q)_{\bil{\rho_k}{\alpha^{\vee}}-k_{\alpha}}}
{(q)_{\bil{\rho_k}{\alpha^{\vee}}}^2},
\end{equation}
where $\rho_k:=\frac{1}{2}\sum_{\alpha>0} k_{\alpha} \alpha$,
$\bil{\cdot}{\cdot}$ is the standard symmetric bilinear form on 
$R$ and $\alpha^{\vee}=2\alpha/(\alpha,\alpha)$ a coroot.

\begin{proposition}
For positive integers $k_{\alpha}$, constant along Weyl orbits,
\begin{multline}\label{Eq_CT-inhom}
\CT\bigg[\prod_{\alpha>0} 
(1-\eup^{-\alpha})(1-t_{\alpha}\eup^{\alpha})
(q\eup^{-\alpha},q\eup^{\alpha})_{k_{\alpha}-1} \bigg] \\
=W(\bt)\prod_{\alpha>0} 
\frac{(q)_{\bil{\rho_k}{\alpha^{\vee}}+k_{\alpha}-1}
(q)_{\bil{\rho_k}{\alpha^{\vee}}-k_{\alpha}}}
{(q)_{\bil{\rho_k}{\alpha^{\vee}}-1}
(q)_{\bil{\rho_k}{\alpha^{\vee}}}}.
\end{multline}
\end{proposition}

\begin{proof}
If we apply \cite[Lemma 4.4]{Cherednik95} with $f$ therein chosen as
$\prod_{\alpha>0}(\eup^{-\alpha},\eup^{\alpha})_{k_{\alpha}}$ and
$k_{\alpha}>0$, then
\[
\CT\bigg[\prod_{\alpha>0}(\eup^{-\alpha},q\eup^{\alpha})_{k_{\alpha}}
\bigg]=\frac{1}{\abs{W}}\, \prod_{\alpha>0} 
\frac{1-q^{\bil{\rho_k}{\alpha^{\vee}}+k_{\alpha}}}
{1-q^{\bil{\rho_k}{\alpha^{\vee}}}}\cdot
\CT\bigg[\prod_{\alpha>0}
(\eup^{-\alpha},\eup^{\alpha})_{k_{\alpha}}\bigg],
\]
where $\abs{W}=d_1\cdots d_r$ is the order of $W$.
Hence \eqref{Eq_Macdonald-q-CT-inhom} for $k_{\alpha}>0$ may be
rewritten as
\begin{equation}\label{Eq_Macdonald-q-CT-inhom-II}
\CT\bigg[\prod_{\alpha>0}
(\eup^{-\alpha},\eup^{\alpha})_{k_{\alpha}} \bigg]
=\abs{W} \prod_{\alpha>0} 
\frac{(q)_{\bil{\rho_k}{\alpha^{\vee}}+k_{\alpha}-1}
(q)_{\bil{\rho_k}{\alpha^{\vee}}-k_{\alpha}}}
{(q)_{\bil{\rho_k}{\alpha^{\vee}}-1}
(q)_{\bil{\rho_k}{\alpha^{\vee}}}}.
\end{equation}
This is of course the $t_{\alpha}=1$ case of \eqref{Eq_CT-inhom}.

Now abbreviate the left-hand side of \eqref{Eq_CT-inhom} by $\CT[\dots]$. 
Then
\[
\CT[\dots]=\CT\bigg[\prod_{\alpha>0}
(\eup^{-\alpha},\eup^{\alpha})_{k_{\alpha}} \,
\frac{1-t_{\alpha}\eup^{\alpha}}{1-\eup^{\alpha}}\bigg]
=\CT\bigg[\prod_{\alpha>0}(\eup^{-\alpha},\eup^{\alpha})_{k_{\alpha}} \,
\frac{1-t_{\alpha}\eup^{-\alpha}}{1-\eup^{-\alpha}}\bigg].
\]
Since acting on the above kernel with the Weyl group does not
affect the constant term, this is also
\[
\CT[\dots]=\frac{1}{\abs{W}}\,\CT\bigg[\prod_{\alpha>0}
(\eup^{-\alpha},\eup^{\alpha})_{k_{\alpha}} \cdot \sum_{w\in W}
\prod_{\alpha>0}
\frac{1-t_{\alpha}\eup^{-w(\alpha)}}{1-\eup^{-w(\alpha)}}
\bigg].
\]
By Macdonald's formula \eqref{Eq_Poincare} the sum over 
$W$ yields $W(\bt)$ so that
\begin{equation}\label{Eq_interm}
\CT[\dots]=\frac{W(\bt)}{\abs{W}}\,\CT\bigg[\prod_{\alpha>0}
(\eup^{-\alpha},\eup^{\alpha})_{k_{\alpha}} \bigg].
\end{equation}
Thanks to \eqref{Eq_Macdonald-q-CT-inhom-II} the proof is done.
\end{proof}

We note that if we set $t_{\alpha}=q^{k_{\alpha}}$ in \eqref{Eq_CT-inhom}
then the constant term on the left coincides with the constant term
in \eqref{Eq_Macdonald-q-CT-inhom}. 
Using the latter identity we thus infer that
\[
W(\bt)|_{t_{\alpha}=q^{k_{\alpha}}}=\prod_{\alpha>0} 
\frac{1-q^{\bil{\rho_k}{\alpha^{\vee}}+k_{\alpha}}}
{1-q^{\bil{\rho_k}{\alpha^{\vee}}}}.
\]
Since this is a rational function identity (polynomial in fact),
$q^{k_{\alpha}}$ may be replaced by $t_{\alpha}$ resulting in
\[
W(\bt)|_{t_{\alpha} \textup{ constant along $W$-orbits}}
=\prod_{\alpha>0} \frac{1-t_{\alpha} 
\prod_{\beta>0} t_{\beta}^{\bil{\beta}{\alpha^{\vee}}/2}}
{1-\prod_{\beta>0} t_{\beta}^{\bil{\beta}{\alpha^{\vee}}/2}}.
\]
Curiously, this product form for a restricted version of the
multivariable Poincar\'e polynomial is slightly  different from 
the one given by Macdonald in \cite[Theorem 2.4]{Macdonald72}:
\[
W(\bt)|_{t_{\alpha} \textup{ constant along $W$-orbits}}
=\prod_{\alpha>0} \frac{1-t_{\alpha} 
\prod_{\beta>0} t_{\beta}^{\bil{\alpha}{\beta^{\vee}}/2}}
{1-\prod_{\beta>0} t_{\beta}^{\bil{\alpha}{\beta^{\vee}}/2}},
\]
although equality of the above two products is readily established.

\section{Theorem~\ref{Thm_Poincare-q-Dyson}}\label{Sec_PqD}

In this section we give a proof Theorem~\ref{Thm_Poincare-q-Dyson} following
the method of the recent proof of the $q$-Dyson conjecture given in 
\cite{KN12}.
We also present a reformulation of the theorem in terms of tournaments,
thus connecting the theorem with results of Bressoud and Goulden.

\subsection{Proof of Theorem~\ref{Thm_Poincare-q-Dyson}}
Write
\[
\CT\bigg[ \prod_{1\leq i<j\leq n}
(x_i/x_j)_{a_i}(qx_j/x_i)_{a_j-1}(1-t_{ij} x_j/x_i)\bigg]=
\sum_{S\subseteq \binom{[n]}{2}}\coef(a;S) t_S,
\]
where the coefficients $\coef(a;S)$ are independent of the $t_{ij}$.
Thus, $\coef(a;S)$ is the coefficient of $\prod_{(i,j)\in S}(x_i/x_j)$ in
the Laurent polynomial  
\[
(-1)^{\abs{S}}\prod_{1\le i<j\le n}(x_i/x_j)_{a_i}(qx_j/x_i)_{a_j-1},
\]
which (recalling the abbreviation $\abs{a}=a_1+\cdots+a_n$) is the same
as the coefficient of the monomial
\[
\prod_{(i,j)\in S}(x_i/x_j)\cdot \prod_{i=1}^n x_i^{\abs{a}-a_i-(n-i)}
\]
in the homogeneous polynomial
\[
F_S(x)=(-1)^{\abs{S}} \prod_{1\le i<j\le n}
\Bigg(\prod_{k=0}^{a_i-1}{\big(x_j-x_iq^k\big)}\cdot
\prod_{k=1}^{a_j-1}{\big(x_i-x_jq^k\big)} \Bigg).
\]
Although the only dependence on $S$ is through the factor $(-1)^{\abs{S}}$,
it has been included to simplify subsequent equations.

To express this coefficient we apply the following simple consequence 
of multivariate Lagrange interpolation, independently formulated by 
Laso\'n \cite{Lason10} and by Karasev and Petrov \cite{KP12}.

\begin{lemma}\label{Lem_interpol} 
Let $\F$ be an arbitrary field and $F\in \F[x_1,x_2,\dots,x_n]$ a
polynomial of degree $\deg(F)\leq d_1+d_2+\cdots+d_n$. 
For arbitrary subsets $A_1,A_2,\dots,A_n$ of $\F$ with 
$\abs{A_i}=d_i+1$, the coefficient of $\prod x_i^{d_i}$ in $F$ is
\[
\sum_{c_1\in A_1} \sum_{c_2\in A_2} \dots \sum_{c_n\in A_n} 
\frac{F(c_1,c_2,\dots,c_n)}{\phi_1'(c_1)\phi_2'(c_2)\cdots\phi_n'(c_n)},
\]
where $\phi_i(z)= \prod_{a\in A_i}(z-a)$.
\end{lemma}

The idea is to apply this lemma with $\F=\Rat(q)$ and $F=F_S$ and 
with a suitable choice of the sets $A_i$ such that $F_S(\bc)\ne 0$ 
holds for at most one element $\bc\in A_1\times\dots\times A_n$.
This will allow us to express the coefficient $\coef(a;S)$ in a simple 
product form if there is an element $w\in\Symm_n$ such that $S=R(w)$ or
else to prove that $\coef(a;S)=0$.
Note that 
\[
\prod_{(i,j)\in S}(x_i/x_j)\cdot
\prod_{i=1}^n x_i^{\abs{a}-a_i-(n-i)}=
\prod_{i=1}^n x_i^{\abs{a}-a_i-(n-i)-d_i+e_i}, 
\]
where $d_j=\abs{\{i:~(i,j)\in S\}}$ and $e_i=\abs{\{j:~(i,j)\in S\}}$.
Clearly $d_i\le i-1$, $e_i\le n-i$ and, recalling that each $a_i>0$, 
we therefore have
\[
0\le \abs{a}-a_i-(n-i)-d_i+e_i\le\abs{a}-a_i. 
\]
Thus, there exist sets $B_i\subseteq\{0,1,\dots,\abs{a}-a_i\}$
such that
\[
\abs{B_i}=\abs{a}-a_i-(n-i)-d_i+e_i+1.
\]
We will construct the $A_i$ in the form
\[
A_i=\{q^{\alpha_i}:~\alpha_i\in B_i\}\subset \F=\Rat(q)
\]
with appropriate choice of the sets $B_i$. 
Before specifying these sets further, we first note that they have the 
right cardinality for a possible application of 
Lemma~\ref{Lem_interpol}.

Assume that $c_i=q^{\alpha_i}\in A_i$ and $F_S(\bc)\ne 0$.
Then  all $\alpha_i$ are distinct. Moreover, $\alpha_i\ge \alpha_j+a_j$
holds for $\alpha_i> \alpha_j$. Consider the unique permutation 
$\pi\in \Symm_n$ for which 
\[
0\le \alpha_{\pi(1)}< \alpha_{\pi(2)}<\dots< \alpha_{\pi(n)}\le 
\abs{a}-a_{\pi(n)}.
\]
Given that 
\[
\abs{a}-a_{\pi(n)}=
\sum_{i=1}^{n-1}a_{\pi(i)}\le 
\sum_{i=1}^{n-1}\big(\alpha_{\pi(i+1)}-\alpha_{\pi(i)}\big)=
\alpha_{\pi(n)}-\alpha_{\pi(1)}\le \abs{a}-a_{\pi(n)},
\]
it follows that $\alpha_{\pi(i)}=a_{\pi(1)}+\cdots+a_{\pi(i-1)}$ holds for
every $1\le i\le n$.

Consider $\ell_i:=(n-i)+d_i-e_i$. Then $0\le \ell_i\le n-1$.
Denote by $K=K(S)$ the smallest nonnegative integer that does not occur
exactly once among the numbers $\ell_1,\dots,\ell_n$. More formally, $K$
is the largest nonnegative integer $k$ for which
\[
\abs{\{i:~\ell_i=j\}}=1
\]
holds for every integer $0\le j<k$. Such a $K\le n$ clearly exists.
Our construction of the appropriate sets $B_i$ (and hence $A_i$) 
relies on the following lemma, stating that $K$ is not among the $\ell_i$.

\begin{lemma}\label{Lem_elli}
$\abs{\{i:~\ell_i=K\}}=0$.
\end{lemma}

\begin{proof}
The proof is by induction on $n$. For the initial step $n=2$ we have
$\ell_2=0$, $\ell_1=1$ if $S=\emptyset$ and $\ell_1=0$, $\ell_2=1$ if 
$S=\{(1,2)\}$. In both cases $K=2$ and the conclusion trivially holds.
So assume that $n\ge 3$ and the statement has been proven for smaller
values of $n$. First consider the case $K=0$, meaning
$\abs{\{i:~\ell_i=0\}}\ne 1$. 
It is enough to show that $\abs{\{i:~\ell_i=0\}}<2$. 
Suppose that on the contrary, $\ell_i=\ell_j=0$ holds for some 
$1\le i<j\le n$.
This means that $d_i=d_j=0$, $e_i=n-i$ and $e_j=n-j$. 
However, if $(i,j)\in S$, then $d_j>0$, whereas if $(i,j)\not\in S$,
then $e_i<n-i$, a contradiction.

Turning to the general case $K>0$, by definition there is a unique 
$1\le \alpha\le n$ such that $\ell_\alpha=0$, whence $d_\alpha=0$
and $e_\alpha=n-\alpha$. 
Construct $S'\subseteq \binom{[n-1]}{2}$ induced by $S$ as follows.
First, for $1\le i<j<\alpha$ let $(i,j)\in S'$ if and only if 
$(i,j)\in S$. 
Next, for $1\le i<\alpha\le j\le n-1$ let $(i,j)\in S'$ if and only if 
$(i,j+1)\in S$.
Finally, for $\alpha\le i<j\le n-1$ let $(i,j)\in S'$ if and only if 
$(i+1,j+1)\in S$.
For this new set $S'$ we have
\begin{align*}
d'_j&=\abs{\{i:~(i,j)\in S'\}}
=d_{j+\chi(j\ge \alpha)}-\chi(j\ge \alpha) \\
\intertext{and}
e'_i&=\abs{\{j:~(i,j)\in S'\}}=e_{i+\chi(i\ge \alpha)},
\end{align*}
where the indicator $\chi(\mathcal{T})$ is $1$ if $\mathcal{T}$ is true 
and $0$ otherwise.
Consequently,
\[
\ell'_i=\big((n-1)-i\big)+d'_i-e'_i=\ell_{i+\chi(i\ge \alpha)}-1
\]
and
\[
K'=\max\big\{k\ge 0:~\abs{\{i:~\ell'_i=j\}}=1
\text{ for every } 0\le j<k\big\}=K-1\le n-1.
\]
It follows from the induction hypothesis that 
$\abs{\{i:~\ell'_i=K'\}}=0$.
Since $\ell_\alpha=0\ne K$, this is equivalent to the statement 
$\abs{\{i:~\ell_i=K\}}=0$.
\end{proof}

Now define the sets $B_i$ satisfying 
$B_i\subseteq\{0,1,\dots,\abs{a}-a_i\}$ and  
$\abs{B_i}=\abs{a}-a_i-\ell_i+1$ as follows. 
For $k=1,2,\dots,K$ write $\tau(k)$ for the unique integer $i$ for which
$\ell_i=k-1$.
Let $B_{\tau(1)}=\{0,1,\dots,\abs{a}-a_{\tau(1)}\}$. 
For $2\le i\le K$ let 
\[
B_{\tau(i)}=\{0,1,\dots,\abs{a}-a_{\tau(i)}\}\setminus
\bigg\{\abs{a}-a_{\tau(i)}-
\sum_{j=1}^k a_{\tau(j)}:~0\le k\le i-2\bigg\}.
\]
Finally, if $i\not\in \{\tau(1),\dots,\tau(K)\}$ then $\ell_i\ge K+1$
is implied by Lemma~\ref{Lem_elli} and therefore we may choose $B_i$ as
an arbitrary $(\abs{a}-a_i-\ell_i+1)$-element subset of
\[
\{0,1,\dots,\abs{a}-a_{\tau(i)}\}\setminus 
\bigg\{\abs{a}-a_{\tau(i)}-
\sum_{j=1}^k a_{\tau(j)}:~0\le k\le K \bigg\}.
\]
Sets $B_i$ thus defined, now consider 
$A_i=\{q^{\alpha_i}:~\alpha_i\in B_i\}$. 
Assume that $c_i=q^{\alpha_i}\in A_i$ and $F_S(\bc)\ne 0$.
As we have seen, for such nonvanishing $\bc$ there is a unique 
permutation $\pi=\pi_{\bc}\in\Symm_n$ such that
\[
\alpha_{\pi(i)}=a_{\pi(1)}+\cdots+a_{\pi(i-1)}
=\abs{a}-a_{\pi(i)}-a_{\pi(i+1)}-\cdots-a_{\pi(n)}
\]
holds for every $1\le i\le n$. 
We will show that such a $\bc$ exists if and only if $K=n$ and 
$\pi(i)=\tau(n-i+1)$ for every $1\le i\le n$. 
First of all, given that 
$\alpha_{\pi(n)}=\abs{a}-a_{\pi(n)}\in B_{\pi(n)}$ it follows that 
$K\ge 1$ and $\pi(n)=\tau(1)$.
Suppose that for some $1\le k<n$ it has been already verified that 
$K\ge k$ and that $\pi(n-i+1)=\tau(i)$ holds for every $1\le i\le k$.
Consider
\[
\alpha_{\pi(n-k)}=\abs{a}-a_{\pi(n-k)}-a_{\pi(n-k+1)}-\cdots-a_{\pi(n)}
=\abs{a}-a_{\pi(n-k)}-a_{\tau(1)}-\cdots-a_{\tau(k)}.
\]
Given that $\alpha_{\pi(n-k)}\in B_{\pi(n-k)}$ it follows that 
$K\ge k+1$ and $\pi((n+1)-(k+1))=\pi(n-k)=\tau(k+1)$.
Thus our claim follows by induction on $k$.

In summary, $K=K(S)<n$ implies $F(\bc)= 0$ for all
$\bc\in A_1\times\dots\times A_n$.
In view of Lemma~\ref{Lem_interpol} it is immediate that $\coef(a;S)=0$
unless $K=n$. 
In the latter case there is a unique $\bc\in A_1\times\dots\times A_n$
for which $F_S(\bc)\ne 0$, to which corresponds a unique permutation 
$\pi=\pi_{\bc}=\pi_S\in\Symm_n$.
First we show that $S=R(w)$ for a suitable element $w\in\Symm_n$. 

\begin{lemma}\label{Lem_bijection}
Suppose that $K=K(S)=n$ holds for a set $S\subseteq \binom{[n]}{2}$.
Then there exists a permutation $w\in\Symm_n$ such that $S=R(w)$.
Conversely, for any permutation $w\in\Symm_n$, the set $S=R(w)$
satisfies $K=n$ with $\ell_{w(j)}=n-j$.
\end{lemma}

\begin{proof}
We prove the first statement by induction on $n$.
Recall that $K=n$ holds if and only if 
$\{\ell_1,\ell_2,\dots,\ell_n\}=\{0,1,\dots,n-1\}$.
The initial case $n=2$ is obvious: $\ell_1=0,\ell_2=1$ if and only if 
$S=R(21)$ whereas $\ell_2=0,\ell_1=1$ if and only if $S=R(12)$.
Next assume that $n\ge 3$ and that we have already established
the claim for smaller values of $n$. 
Let $\tau(1)=\alpha$ and construct the set $S'\subseteq\binom{[n-1]}{2}$
as in the proof of Lemma~\ref{Lem_elli}. 
Then $K'=K(S')=n-1$ and $\{\ell'_1,\dots,\ell'_{n-1}\}=\{0,\dots,n-2\}$, 
so there exists a permutation $w'\in\Symm_{n-1}$ such that $S'=R(w')$.
To construct $w\in\Symm_n$, let $w(n)=\alpha$ and
$w(i)=w'(i)+\chi\big(w'(i)\ge \alpha\big)$.
It is easy to check that $S=R(w)$. 

For the identity permutation $w=\id$ we have $S=R(w)=\emptyset$, 
$d_j=e_j=0$ and thus $\ell_{w(j)}=\ell_j=n-j$ for every $1\le j\le n$.
So to verify the converse statement it is enough to show that if it is
true for some permutation $w\in\Symm_n$, then it also holds for the
permutation $w'$ obtained from $w$ by the transposition of $w(i)$ and 
$w(i+1)$ for some $1\le i\le n-1$.
Obviously, $\ell'_{w'(j)}=\ell'_{w(j)}=\ell_{w(j)}=n-j$ holds for 
$j\not\in \{i,i+1\}$.
Next we verify $\ell'_{w'(i)}=n-i$. 
If $w(i)<w(i+1)$, then indeed
\begin{align*}
\ell'_{w'(i)}&=\big(n-w'(i)\big)+d'_{w'(i)}-e'_{w'(i)}\\
&=(n-w(i+1))+(d_{w(i+1)}+1)-e_{w(i+1)}\\
&=\ell_{w(i+1)}+1\\
&=n-(i+1)+1.
\end{align*}
A similar argument works for the case $w(i)>w(i+1)$, which we leave to
the reader along with the verification of $\ell'_{w'(i+1)}=n-i-1$.
\end{proof}

Consider $\mathcal{S}:=\big\{S\subseteq \binom{[n]}{2}:~K(S)=n\big\}$. 
It follows from Lemma~\ref{Lem_bijection} that the map $w\to R(w)$
defines a bijection from $\Symm_n$ to $\mathcal{S}$.
In view of $\ell_{\tau(k)}=k-1$ and $\pi(n-i+1)=\tau(i)$, its inverse is
given by $S\to \pi_S$.
Thus, in order to complete the proof of 
Theorem~\ref{Thm_Poincare-q-Dyson} it only remains to show that
\[
\coef\big(a;R(\pi)\big)=
\prod_{i=1}^n \qbin{\A_i-1}{a_i-1}
\frac{1-q^{\A_i}}{1-q^{\pi(\A_i)}}.
\] 
By \eqref{Eq_qmultinomial} this is equivalent to showing that
\[
\coef\big(a;R(\pi)\big)=\qbin{\abs{a}}{a}
\prod_{i=1}^n \frac{1-q^{a_i}}{1-q^{a_{\pi(1)}+\cdots+a_{\pi(i)}}}.
\] 
According to Lemma~\ref{Lem_interpol},
\[
\coef\big(a;R(\pi)\big)=\frac{F_{R(\pi)}(c_1,c_2,\dots,c_n)}
{\phi_1'(c_1)\phi_2'(c_2)\cdots\phi_n'(c_n)},
\]
where $\phi_i(z)= \prod_{a\in A_i}(z-a)$ and $c_i=q^{\alpha_i}$, with
$\alpha_{\pi(i)}=a_{\pi(1)}+\cdots+a_{\pi(i-1)}$.
Define $s_1,\dots,s_{n+1}$ by
\begin{gather*}
s_i:=\alpha_{\pi(i)}=a_{\pi(1)}+\cdots+a_{\pi(i-1)}=
\abs{a}-a_{\pi(i)}-a_{\pi(i+1)}-\cdots-a_{\pi(n)}, \\[2mm]
s=s_{n+1}:=a_{\pi(1)}+\cdots+a_{\pi(n)}=\abs{a}
\end{gather*}
so that $s_{i+1}=s_i+a_{\pi(i)}$.
In view of the definitions of $A_i$, $B_{\tau(i)}$ and the relation 
$\pi(i)=\tau(n-i+1)$ we have 
\[
\phi_{\pi(i)}(z)=
\frac{\displaystyle{\prod_{l=0}^{\abs{a}-a_{\pi(i)}}\big(z-q^l\big)}}
{\displaystyle{\prod_{j=i+2}^{n+1}\Big(z-q^{\abs{a}-a_{\pi(i)}-
(a_{\pi(j)}+a_{\pi(j+1)}+\cdots+a_{\pi(n)})}\Big)}}.
\]
Therefore
\begin{align*}
\phi_{\pi(i)}'(q^{\alpha_{\pi(i)}})&=
\frac{\displaystyle{\prod_{l=0}^{\alpha_{\pi(i)}-1}
\big(q^{\alpha_{\pi(i)}}-q^l\big)}
\cdot \displaystyle{\prod_{l=\alpha_{\pi(i)}+1}^{s-a_{\pi(i)}}
\big(q^{\alpha_{\pi(i)}}-q^l\big)}}
{\displaystyle{\prod_{j=i+2}^{n+1}\Big(q^{\alpha_{\pi(i)}}-
q^{s-a_{\pi(i)}-(a_{\pi(j)}+a_{\pi(j+1)}+\cdots+a_{\pi(n)})}\Big)}}\\
&=\frac{\displaystyle{\prod_{l=0}^{s_i-1}q^l\big(q^{s_i-l}-1\big)}
\cdot \displaystyle{\prod_{l=s_i+1}^{s-a_{\pi(i)}}q^{s_i}
\big(1-q^{l-s_i}\big)}}
{\displaystyle{\prod_{j=i+1}^n q^{\alpha_{\pi(i)}}\Big(1-
q^{a_{\pi(i+1)}+a_{\pi(i+2)}+\cdots+a_{\pi(j)}}\Big)}}\\
&=\frac{(-1)^{s_i}q^{t_i}(q)_{s_i}(q)_{s-s_{i+1}}}
{\displaystyle{\prod_{j=i+1}^n\big(1-q^{s_{j+1}-s_{i+1}}\big)}}
\end{align*}
with $t_i=\binom{s_i}{2}+s_i(s-s_{i+1})-(n-i)s_i$. 
Next we split $F_{R(\pi)}(q^{\alpha_1},q^{\alpha_2},\dots,q^{\alpha_n})$ 
into the product $(-1)^{\abs{R(\pi)}}\Pi_< \, \Pi_>$, where
\begin{align*}
\Pi_<&:=\prod_{\substack{\pi(u)<\pi(v)\\u<v}}
\bigg(\,\prod_{k=0}^{a_{\pi(u)}-1}\big(q^{\alpha_{\pi(v)}}-
q^{\alpha_{\pi(u)}+k}\big)\cdot
\prod_{k=1}^{a_{\pi(v)}-1}\big(q^{\alpha_{\pi(u)}}-
q^{\alpha_{\pi(v)}+k}\big) \bigg)\\
&=\prod_{\substack{\pi(u)<\pi(v)\\u<v}}
\bigg(\,\prod_{k=0}^{a_{\pi(u)}-1}q^{s_u+k}\big(q^{s_v-s_u-k}-1\big)
\cdot \prod_{k=1}^{a_{\pi(v)}-1}q^{s_u}\big(1-q^{s_v-s_u+k}\big) \bigg),
\end{align*}
and
\begin{align*}
\Pi_>&=\prod_{\substack{\pi(u)<\pi(v)\\u>v}}
\bigg(\,\prod_{k=0}^{a_{\pi(u)}-1}\big(q^{\alpha_{\pi(v)}}-
q^{\alpha_{\pi(u)}+k}\big)\cdot
\prod_{k=1}^{a_{\pi(v)}-1}\big(q^{\alpha_{\pi(u)}}-
q^{\alpha_{\pi(v)}+k}\big) \bigg)\\
&=\prod_{\substack{\pi(u)<\pi(v)\\u>v}}
\bigg(\,\prod_{k=0}^{a_{\pi(u)}-1}q^{s_v}\big(1-q^{s_u-s_v+k}\big)
\cdot \prod_{k=1}^{a_{\pi(v)}-1}q^{s_v+k}\big(q^{s_u-s_v-k}-1\big)
\bigg).
\end{align*}
Thus,
\begin{align*}
\Pi_<&=(-1)^{s_{<}}q^{t_{<}}\prod_{\substack{i<j \\ \pi(i)<\pi(j)}}
\bigg(\frac{(q)_{s_j-s_i}}{(q)_{s_j-s_{i+1}}}\cdot
\frac{(q)_{s_{j+1}-s_i-1}}{(q)_{s_j-s_i}}\bigg)\\
&=(-1)^{s_{<}}q^{t_{<}}\prod_{\substack{i<j \\ \pi(i)<\pi(j)}}
\bigg(\frac{1}{1-q^{s_{j+1}-s_i}}\cdot
\frac{(q)_{s_{j+1}-s_i}}{(q)_{s_j-s_{i+1}}}\bigg)
\end{align*}
holds with 
\begin{align*}
s_{<}&=\sum_{\substack{i<j \\ \pi(i)<\pi(j)}}a_{\pi(i)}, &
t_{<}&=\sum_{\substack{i<j \\ \pi(i)<\pi(j)}}
\bigg( \binom{a_{\pi(i)}}{2}+(a_{\pi(i)}+a_{\pi(j)}-1)s_i \bigg). \\
\intertext{Similarly, with the notation}
s_{>}&=\sum_{\substack{i<j \\ \pi(i)>\pi(j)}}(a_{\pi(i)}-1), &
t_{>}&=\sum_{\substack{i<j \\ \pi(i)>\pi(j)}}\bigg(
\binom{a_{\pi(i)}}{2}+(a_{\pi(i)}+a_{\pi(j)}-1)s_i \bigg)
\end{align*}
we can rewrite $\Pi_>$ as 
\[
\Pi_{>}=(-1)^{s_{>}}q^{t_{>}}\prod_{\substack{i<j \\ \pi(i)>\pi(j)}}
\bigg(\frac{1}{1-q^{s_{j+1}-s_i}}\cdot
\frac{(q)_{s_{j+1}-s_i}}{(q)_{s_j-s_{i+1}}}\bigg).
\]
Consequently,
\begin{align*}
F_{R(\pi)}&(q^{\alpha_1},q^{\alpha_2},\dots,q^{\alpha_n})\\
&=(-1)^{\abs{R(\pi)}+s_{<}+s_{>}}q^{t_{<}+t_{>}}
\prod_{1\le i<j\le n}\bigg(\frac{1}{1-q^{s_{j+1}-s_i}}\cdot
\frac{(q)_{s_{j+1}-s_i}}{(q)_{s_j-s_{i+1}}}\bigg)\\
&=(-1)^{\sum_{i=1}^n s_i} q^{t_{<}+t_{>}}
\prod_{1\le i<j\le n}\frac{1}{1-q^{s_{j+1}-s_i}}\cdot
\prod_{i=1}^n\frac{(q)_{s_i}(q)_{s-s_i}}{(q)_{s_{i+1}-s_i}} ,
\end{align*}
where we have used that $s_{n+1}=s$, $s_1=0$ and
$\abs{R(\pi)}+s_{<}+s_{>}=\sum_{i<j} a_{\pi(i)}=
\sum_{j=1}^n s_j$.
It thus follows that in the expression
\[
\coef\big(a;R(\pi)\big)=
\frac{F_{R(\pi)}(q^{\alpha_1},q^{\alpha_2},\dots,q^{\alpha_n})}
{\phi_1'(q^{\alpha_1})\phi_2'(q^{\alpha_2})\cdots \phi_n'(q^{\alpha_n})}
\]
the powers of $-1$ cancel out. So do the powers of $q$, according to
the following observation.

\begin{claim}
$t_{<}+t_{>}=\sum_{i=1}^n t_i$.
\end{claim}

\begin{proof}
Using
\[
\sum_{1\leq i<j\leq n} a_{\pi(j)}s_i=
\sum_{i=1}^n s_i\sum_{j=i+1}^n (s_{j+1}-s_j)=
\sum_{i=1}^n s_i(s-s_{i+1}),
\]
eliminating all other $a_{\pi(i)}$ by $a_{\pi(i)}=s_{i+1}-s_i$ and
finally using $\binom{a-b}{2}=\binom{a}{2}+\binom{b}{2}+(1-a)b$ 
we find
\[
t_{<}+t_{>}=\sum_{i=1}^n
\bigg[ (n-i)\bigg(\binom{s_{i+1}}{2}-\binom{s_i}{2}\bigg)
-(n-i)s_i+s_i(s-s_{i+1})\bigg].
\]
Recalling that $s_1=0$, the sum over the binomial coefficients telescopes to 
$\sum_i \binom{s_i}{2}$, establishing the claim.
\end{proof}

\noindent
Finally,
\begin{align*}
\coef(a; R(\pi))&=\frac{\displaystyle \prod_{1\le i<j\le n}
\frac{1}{1-q^{s_{j+1}-s_i}}\cdot\prod_{i=1}^n
\frac{(q)_{s_i}(q)_{s-s_i}}{(q)_{s_{i+1}-s_i}}}
{\displaystyle{\prod_{i=1}^n \frac{(q)_{s_i}(q)_{s-s_{i+1}}}
{\displaystyle{\prod_{j=i+1}^n\big(1-q^{s_{j+1}-s_{i+1}}\big)}}}}\\
&=\prod_{1\le i<j\le n}\frac{1-q^{s_{j+1}-s_{i+1}}}{1-q^{s_{j+1}-s_i}}
\cdot\prod_{i=1}^n\frac{(q)_{s-s_i}}{(q)_{s-s_{i+1}}(q)_{s_{i+1}-s_i}}\\
&=\prod_{i=1}^n\frac{1-q^{s_{i+1}-s_i}}{1-q^{s_{i+1}}}\cdot
\frac{(q)_s}{\displaystyle{\prod_{i=1}^n(q)_{a_{\pi(i)}}}}\\
&=\qbin{\abs{a}}{a}
\prod_{i=1}^n \frac{1-q^{a_i}}{1-q^{a_{\pi(1)}+\cdots+a_{\pi(i)}}}
\end{align*}
as required.

\subsection{Tournaments}

Recall from the introduction that $I(w)$ denotes the inversion set of a
permutation $w$ and $R(w)=I(w^{-1})$. In particular, $R(w)$ contains
all pairs of integers in the permutation $w$ that are not in natural order.
For example, if $w=(3,4,6,1,5,7,2)$ then 
\begin{align*}
I(w)&=\{(1,4),(1,7),(2,4),(2,7),(3,4),(3,5),(3,7),(5,7),(6,7)\}, \\
R(w)&=\{(1,3),(1,4),(1,6),(2,3),(2,4),(2,5),(2,6),(2,7),(5,6)\}.
\end{align*}

Now let $T$ be a tournament on $n$ vertices \cite{Berge71,Stanley97}.
That is, $T$ is a directed complete graph on $n$ labelled vertices 
$1,\dots,n$ (thought of as players).
The interpretation of a directed edge from vertex $i$ to 
vertex $j$ (also written as $i\to j$) is that (player) $i$ beats (player) $j$.
If $i\to j$ we also write $(i,j)\in T$.
A tournament is transitive if $i\to j$ and $j\to k$ implies
that $i\to k$. The winner permutation $w_{\textup{win}}$ 
of a transitive tournament
is a ranking of the vertices (players) from best to worst.
The set $R(w_{\textup{win}})$ precisely contains 
those edges of $T$ which have been reversed 
compared to the edges of the tournament $1\to 2\to \cdots \to n$.

To reformulate Theorem~\ref{Thm_Poincare-q-Dyson} in terms
of tournaments we repeat the expansion \eqref{Eq_expansion}. 
That is, we use
\[
\prod_{1\leq i<j\leq n}
(1-t_{ij} x_j/x_i) =\sum_{S\subseteq \binom{[n]}{2}}(-1)^{\abs{S}} t_S
\prod_{(i,j)\in S} x_j/x_i,
\]
with $t_S=\prod_{(i,j)\in S} t_{ij}$. 
Then equating coefficients of $t_S$ in \eqref{Eq_Poincare-q-Dyson} and
using that
\[
x_j/x_i \,
(x_i/x_j)_{a_i}(qx_j/x_i)_{a_j-1}=-(x_j/x_i)_{a_j}(qx_i/x_j)_{a_i-1}
\]
yields, for any $S\subseteq \binom{[n]}{2}$,
\begin{multline*}
\CT\bigg[ 
\prod_{(i,j)\in S} (x_j/x_i)_{a_j}(qx_i/x_j)_{a_i-1}
\prod_{(i,j)\in \bar{S}} (x_i/x_j)_{a_i}(qx_j/x_i)_{a_j-1}
\bigg] \\
=\begin{cases}\displaystyle \qbin{\An}{a}
\prod_{i=1}^n \frac{1-q^{a_i}}{1-q^{w(\A_i)}} 
& \text{if $S=R(w)$ for some $w\in\Symm_n$,} \\[4mm]
0 & \text{otherwise},
\end{cases}
\end{multline*}
where $\bar{S}=\binom{[n]}{2}\setminus S$.
If we define the tournament $T$ as $T=\bar{S}\cup \{(j,i):~(i,j)\in S\}$
then the left-hand side can be written as
\[
\CT\bigg[\prod_{(i,j)\in T} (x_i/x_j)_{a_i}(qx_j/x_i)_{a_j-1}\bigg].
\]
To summarise, Theorem~\ref{Thm_Poincare-q-Dyson} is equivalent to 
the following statement about tournaments.
\begin{theorem}
Let $T$ be a tournament and $a=(a_1,\dots,a_n)$ a sequence
of positive integers. Then
\begin{multline*}
\CT\bigg[\prod_{(i,j)\in T} (x_i/x_j)_{a_i}(qx_j/x_i)_{a_j-1} \bigg] \\
=\begin{cases}\displaystyle \qbin{\An}{a}
\prod_{i=1}^n \frac{1-q^{a_i}}{1-q^{w_{\textup{win}}(\A_i)}} & 
\text{if $T$ is transitive}, \\[4mm]
0 & \text{if $T$ is nontransitive}.
\end{cases}
\end{multline*}
\end{theorem}
The `nonzero part' of this theorem is \cite[Theorem 2.2]{BG85}
and the `zero part' is \cite[Theorem 1.3]{BG85}, which is Bressoud and
Goulden's `Master Theorem'.

\section{Kadell's conjecture}\label{Sec_Kadell}

In this section we prove Theorem~\ref{Thm_Kadell-conjecture} 
and derive a number of generalisations involving Schur functions.

\subsection{A scalar product}
To make our proof of \eqref{Eq_Kadell-nul} and
 \eqref{Eq_Kadell-niet-nul} reasonably self-contained, we begin by 
reviewing some basic facts about two families of polynomials known as
type $\mathrm{A}$ Demazure characters or key polynomials, see e.g., 
\cite{Demazure72,FL09,Ion03,LS90}.

Throughout this section we assume that $x=(x_1,\dots,x_n)$.
For $1\leq i\leq n-1$ let 
$\pi_i$ and $\hat{\pi}_i=\pi_i-\id$ be the isobaric divided 
difference operators
\[
(\pi_i f)(x)=
\frac{x_if(x)-x_{i+1}f(\dots,x_{i+1},x_i,\dots)}{x_i-x_{i+1}}.
\]
and
\[
(\hat{\pi}_i f)(x)=
\frac{x_{i+1}f(x)-x_{i+1}f(\dots,x_{i+1},x_i,\dots)}{x_i-x_{i+1}}.
\]
Recall that $\Comp_n$ denotes the set of compositions of the form
$v=(v_1,\dots,v_n)$.
If $v$ is a partition, i.e., $v=v^{+}$, we say that $v$ is dominant.
We denote by $\bar{v}$ the composition $\bar{v}=(v_n,\dots,v_1)$. 
If $\bar{v}=v^{+}$ we say that $v$ is antidominant.
For example, if $v=(1,0,4,1,0,3)$ then $v^{+}=(4,3,1,1)$ and 
$\bar{v}=(3,0,1,4,0,1)$.

Let $s_i v=(\dots,v_{i+1},v_i,\dots)$. 
Then the key polynomials $K_v(x)$ and $\hat{K}_v(x)$ for $v\in\Comp_n$ are
defined by the recursion
\[
K_{s_iv}=\pi_i K_v,\qquad
\hat{K}_{s_iv}=\hat{\pi}_i \hat{K}_v,
\]
subject to the initial conditions
\[
K_v(x)=\hat{K}_v(x)=x^v \quad \text{if $v$ is dominant}.
\]
This definition is consistent since both types of isobaric divided
difference operators satisfy the braid relations. (In fact, both
$\{-\pi_i\}_{1\leq i<n}$ and $\{\hat{\pi}_i\}_{1\leq i<n}$ 
generate the $0$-Hecke algebra of the symmetric group.)
For $v$ an antidominant composition $K_v$ corresponds to a Schur 
function. 
Specifically $K_{\bar{\la}}=s_{\la}$ for $\la\in\Part_n$.

The significance of the key polynomials to constant term identities 
rests with the fact that they form adjoint bases of the ring 
$\Z[x_1,\dots,x_n]$ with respect to the scalar product
\begin{equation}\label{Eq_Scalar-product}
\ip{f}{g}=\CT\Big[ f(x_1,\dots,x_n)g(x_n^{-1},\dots,x_1^{-1})
\prod_{1\leq i<j\leq n}(1-x_i/x_j) \Big].
\end{equation}
More precisely,
\[
\ip{K_v}{\hat{K}_w}=\delta_{v,\bar{w}}.
\]
This in particular implies that
\[
\ip{s_{\la}}{x^{\mu}}=\delta_{\la\mu}.
\]
We later require the generalisation of this formula to the case
of nondominant monomials:
\begin{equation}\label{Eq_Schur-Monomial}
\ip{s_{\la}}{x^v}=
\begin{cases}
(-1)^{l(w)} & \text{if $v+\delta=w(\la+\delta)$ for some $w\in\Symm_n$} \\[2mm]
0 & \text{otherwise},
\end{cases}
\end{equation}
where $\delta:=(n-1,\dots,1,0)$.

\subsection{Towards Kadell}\label{Sec_Towards-Kadell}
Combining simple properties of the scalar product
\eqref{Eq_Scalar-product} with Theorem~\ref{Thm_Poincare-q-Dyson}
we derive a number of Kadell-like constant term identities,
given in Propositions~\ref{Prop_CT-kappa}--\ref{Prop_CT-vnu} below.
Kadell's conjecture follows from these results in a straightforward
manner.

For $x=(x_1,\dots,x_n)$, $a=(a_1,\dots,a_n)$ with all $a_i>0$ and
\begin{equation}\label{Eq_tset}
\bt=\{t_{ij}\}_{1\leq i<j\leq n}
\end{equation}
we define 
\[
D(a;x;\bt)=\prod_{1\leq i<j\leq n}
(x_i/x_j)_{a_i}(qx_j/x_i)_{a_j-1} (1-t_{ij}x_j/x_i),
\]
and 
\[
D_{v,\la}(a;\bt)=\CT\big[x^{-v} s_{\la}\big(x^{(a)}\big) 
D(a;x;\bt)\big],
\]
for $v\in \Comp_n$ and $\la\in\Part_{\An}$.
As before, $x^{(a)}$ is given by \eqref{Eq_xa}.
Note that $D_{v,\la}(a)$ defined in \eqref{Eq_D_vla} corresponds to
$D_{v,\la}(a;\bt)$ for $t_{ij}=q^{a_j}$.
If $\la=v^{+}$, as will often be the case, we simply write $D_v(a;\bt)$:
\[
D_v(a;\bt)=D_{v,v^{+}}(a;\bt)=
\CT\big[x^{-v} s_{v^{+}}\big(x^{(a)}\big) D(a;x;\bt)\big].
\]
With this notation \eqref{Eq_Poincare-q-Dyson} becomes
\[
D_0(a;\bt)
=W_a(\bt)\prod_{i=1}^n \qbin{\sigma_i-1}{a_i-1} ,
\]
where $0$ denotes the composition $(0,\dots,0)$.
Since $D(a;x;\bt)$ is homogeneous of degree zero, 
$D_{v,\la}(a;\bt)=0$ if $\abs{\la}\neq\abs{v}$.

For $w\in\Symm_n$ and $b$ an arbitrary sequence of length $n$
let $w(b)=(b_{w(1)},\dots,b_{w(n)})$. Also define 
$w(\bt)=\{t_{w(i),w(j)}\}_{1\leq i<j\leq n}\big)$ with
$t_{ij}:=1/t_{ji}$ for $i>j$.
For example, if $n=3$ and $w=(2,3,1)$ then
\begin{multline*}
D(w(a);w(x);w(\bt))=(1-t_{23}x_2/x_3)(1-t_{12}^{-1}x_2/x_1)
(1-t_{13}^{-1}x_3/x_1)\\
\times
(x_2/x_3,x_2/x_1)_{a_2}(qx_3/x_2)_{a_3-1} (x_3/x_1)_{a_3}
(qx_1/x_2,qx_1/x_3)_{a_1-1}.
\end{multline*}

\begin{lemma}\label{Lem_sym}
For $w\in\Symm_n$
\[
D_{v,\la}(a;\bt)=t_{R(w)} D_{w(v),\la}\big(w(a);w(\bt)\big),
\]
or, when $\la=v^{+}$,
\[
D_v(a;\bt)=t_{R(w)} D_{w(v)}\big(w(a);w(\bt)\big).
\]
\end{lemma}

\begin{proof}
It is not hard to verify that
\[
D\big(w(a);w(x);w(\bt)\big)=\frac{D(a;x;\bt)}{t_{R(w)}}.
\]
Since $s_{\la}\big(w(x)^{w(a)}\big)=s_{\la}\big(x^{(a)}\big)$
(recall that $s_{\la}$ is a symmetric function), and
$w(x)^{-w(v)}=x^{-v}$ the result follows.
\end{proof}

We now introduce a second set of variables $y=(y_1,\dots,y_m)$ and write
$x,y$ for the concatenation of $x$ and $y$:
$x,y=(x_1,\dots,x_n,y_1,\dots,y_m)$.
We also define $\btau$ as
\begin{equation}\label{Eq_tauset}
\btau=\{t_{ij}\}_{1\leq i<j\leq m+n}|_{t_{ij}=0 \text{ for } i>n}.
\end{equation}
If we finally denote the sequence 
$(a_1,\dots,a_n,\underbrace{1,\dots,1}_{m \text{ times}})$
by $a1^m$, then an elementary calculation shows that
\[
D(a1^m;x,y;\btau)=D(a;x;\bt) 
\prod_{1\leq i<j\leq m}(1-y_i/y_j)
\prod_{i=1}^n
\prod_{j=1}^m(1-s_{ij}y_j/x_i)(x_i/y_j)_{a_i},
\]
where $s_{ij}:=t_{i,j+n}$.
Given a partition $\la$, let $\la'$ be its conjugate.
Using the dual Cauchy identity for Schur functions \cite{Macdonald95}
in the form
\[
\sum_{\la} (-1)^{\abs{\la}} s_{\la}(u_1,\dots,u_n) 
s_{\la'}(v_1,\dots,v_m)=
\prod_{i=1}^n \prod_{j=1}^m (1-u_iv_j)
\]
we can expand 
\begin{align*}
\prod_{i=1}^n \prod_{j=1}^m (x_i/y_j)_{a_i}
&=\prod_{i=1}^n \prod_{j=1}^m \prod_{k=0}^{a_i-1} \big(1-q^kx_i/y_j\big)\\
&=\prod_{i=1}^{\An} \prod_{j=1}^m \big(1-x_i^{(a)}/y_j\big) 
=\sum_{\la} (-1)^{\abs{\la}} s_{\la}\big(x^{(a)}\big) 
s_{\la'}\big(y^{-1}\big),
\end{align*}
where $y^{-1}:=(y_1^{-1},\dots,y_m^{-1})$.
Hence
\begin{multline*}
D(a1^m;x,y;\btau)=D(a;x;\bt)  \\ \times
\sum_{\la} (-1)^{\abs{\la}} 
s_{\la}\big(x^{(a)}\big) s_{\la'}\big(y^{-1}\big)
\prod_{1\leq i<j\leq m}(1-y_i/y_j)
\prod_{i=1}^n \prod_{j=1}^m(1-s_{ij}y_j/x_i).
\end{multline*}
Our next step is to take the constant term with respect to $y$, 
denoted by $\CT_y$.
Recalling the scalar product \eqref{Eq_Scalar-product}, this gives
\begin{multline*}
\CT_y\big[D(a1^m;x,y;\btau) \big]\\
=D(a;x;\bt) \sum_{\la} (-1)^{\abs{\la}} s_{\la}\big(x^{(a)}\big) 
\Big\langle\prod_{i=1}^n \prod_{j=1}^m
(1-s_{ij}y_j/x_i),s_{\la'}(y)\Big\rangle_{\! y},
\end{multline*}
where we have written $\ip{\cdot}{\cdot}_y$ to indicate that the 
scalar product is with respect to $y$ instead of the usual $x$-variables.

Let $\Mat_{n,m}$ be the set of $(0,1)$-matrices of size $n\times m$
and $\Comp_{n,m}=\{(v_1,\dots,v_n)\in\Z^n:~0\leq v_i\leq m 
\text{ for all $1\leq i\leq n$} \}$ the subset of the set of
compositions $\Comp_n$ that fit in a rectangle of size $n\times m$. 
For $\kappa\in\Mat_{n,m}$ denote by $r(\kappa)\in\Comp_{n,m}$ and 
$c(\kappa)\in\Comp_{m,n}$ the compositions encoding the
row and column sums of $\kappa$:
\[
r_i(\kappa)=\sum_{j=1}^m \kappa_{ij}
\quad\text{and}\quad
c_j(\kappa)=\sum_{i=1}^n \kappa_{ij}.
\]
We shall write $\abs{\kappa}=\sum_{ij}\kappa_{ij}=\abs{r(\kappa)}
=\abs{c(\kappa)}$
for the sum of the entries of $\kappa$.
For example, if
\[
\kappa=\begin{pmatrix} 
0 & 1 & 0 & 0 & 1 \\
1 & 1 & 0 & 0 & 0 \\
1 & 1 & 0 & 1 & 1 \\
0 & 0 & 0 & 1 & 0
\end{pmatrix}
\in\Mat_{4,5}
\]
then $r(\kappa)=(2,2,4,1)$, $c(\kappa)=(2,3,0,2,2)$ and
$\abs{\kappa}=9$.

Let 
\[
\bs:=\{s_{ij}\}_{1\leq i\leq n;~1\leq j\leq m}=
\{t_{i,j+n}\}_{1\leq i\leq n;~1\leq j\leq m},
\]
and for $f(\bs)$ a polynomial in $\bs$ denote by $[\bs^{\kappa}] f(\bs)$
the coefficient of the monomial 
$\bs^{\kappa}:=\prod_{i=1}^n\prod_{j=1}^m s_{ij}^{\kappa_{ij}}$ in $f$.
Since
\[
\prod_{i=1}^n \prod_{j=1}^m (1-s_{ij}y_j/x_i)
=\sum_{\kappa\in \Mat_{n,m}} (-1)^{\abs{\kappa}} 
y^{c(\kappa)}x^{-r(\kappa)} \bs^{\kappa},
\]
we have
\begin{multline*}
\big[\bs^{\kappa}\big] \CT_y\big[D(a1^m;x,y;\btau) \big]\\
=D(a;x;\bt) \sum_{\la} (-1)^{\abs{\la}+\abs{\kappa}} x^{-r(\kappa)}
s_{\la}\big(x^{(a)}\big) 
\big\langle y^{c(\kappa)},s_{\la'}(y)\big\rangle_{\! y}.
\end{multline*}
Recalling \eqref{Eq_Schur-Monomial}, the summand on the right vanishes
unless $\la$ is the (unique) partition such that its conjugate satisfies
$\la'=w\big(c(\kappa)+\delta_m\big)-\delta_m$, where 
$\delta_m:=(m-1,\dots,1,0)$ and $w\in\Symm_m$. 
This in particular implies that 
$\abs{\la}=\abs{c(\kappa)}=\abs{\kappa}$.
Assuming such $\la$ and $w$, and interchanging the left and right-hand 
sides, we obtain
\[
x^{-r(\kappa)} s_{\la}\big(x^{(a)}\big) D(a;x;\bt)
=(-1)^{l(w)} \big[\bs^{\kappa} \big]
\CT_y \big[D(a1^m;x,y;\btau)\big].
\]
As a final step we take the constant term with respect to $x$
to arrive at the following result.
\begin{proposition}\label{Prop_CT-kappa}
Let $a=(a_1,\dots,a_n)$ be a sequence of positive integers,  
$\delta_m=(m-1,\dots,1,0)$, and let $\bt$ and
$\btau$ be given by \eqref{Eq_tset} and \eqref{Eq_tauset} respectively.
For $\kappa$ a $(0,1)$-matrix in $\Mat_{n,m}$, $\la\in\Comp_{\An,m}$
a partition of $\abs{\kappa}$ and $w\in\Symm_m$ such that
\begin{equation}\label{Eq_law}
\la'=w\big(c(\kappa)+\delta_m\big)-\delta_m,
\end{equation}
we have
\[
D_{r(\kappa),\la}(a;\bt)
=(-1)^{l(w)} \big[\bs^{\kappa}\big] D_0(a1^m;\btau).
\]
\end{proposition}
We should remark that one cannot choose $\kappa\in\Mat_{n,m}$
such that all conceivable row and column sums $r(\kappa)\in\Comp_{n,m}$ 
and $c(\kappa)\in\Comp_{m,n}$ arise.
Apart from the obvious restriction 
$\abs{r(\kappa)}=\abs{c(\kappa)}$, we are bound by the Gale--Ryser
theorem \cite{BR91}. 
This theorem for example says that a pair of partitions $\mu$ and $\nu$
can arise as the row- and column-sums of a $(0,1)$-matrix if and only if
$\mu$ is dominated by $\nu'$.
The smallest example of a constant term that does not occur in the above
is $n=2$, $\la=(1,1)$ and $r(\kappa)=(2,0)$ or $(0,2)$. Whatever choice of 
$m\geq 2$ we make, $\max\{\la'+\delta_m\}=m+1$ whereas 
$\max\{c(\kappa)+\delta_m\}\leq m$. Hence there is no $w\in\Symm_m$
such that \eqref{Eq_law} holds.
Despite these caveats, the above proposition is very useful since for 
each admissible pair $\la,r(\kappa)$ the constant term on the right-hand
side has already been evaluated by Theorem~\ref{Thm_Poincare-q-Dyson}
with $n\mapsto m+n$.
In particular,
\[
D_0(a1^m;\btau)=
\sum_{w\in \Symm_{m+n}} c_w(a) t_{R(w)}\big|_{t_{ij}=0 \text{ for } i>n}.
\]
The fact that $t_{ij}=0$ for $i>n$ means that only the
$(m+n)!/m!$ permutations of $(1,\dots,n+m)$ contribute for
which each pair of integers in $n+1,\dots,n+m$ occurs 
in natural order.
But this in turn means that if $t_{i,j+n}=s_{ij}$ is in $t_{R(w)}$ then so
must be $s_{i1},s_{i2},\dots,s_{i,j-1}$; if $i<n$ overtakes $j>n$ but 
$n+1,\dots,n+m$ are in natural order then $i$ must also have overtaken
$n+1,\dots,j-1$.
In terms of the $(0,1)$-matrices this means that each row must be a
sequence of ones (possibly of zero length) followed by a sequence of
zeros (possibly of zero length), otherwise the coefficient of 
$\bs^{\kappa}$ is necessarily zero.
We summarise this in our next proposition.
\begin{proposition}\label{Prop_CT-zero}
Assume the conditions of Proposition~\ref{Prop_CT-kappa}.
Then
\[
D_{r(\kappa),\la}(a;\bt)=0
\]
if $\kappa$ is not a $(0,1)$-matrix such that in each row all ones precede 
all zeros.
\end{proposition}

A $(0,1)$-matrix $\kappa$ such that the ones in each row precede
the zeros is uniquely determined by its row sums,
and in particular 
$c_j(\kappa)=\abs{\{r_i(\kappa):~i\geq j\}}$
or, more succinctly, $c(\kappa)=(r^{+}(\kappa))'=:\nu'$.
For example, for
\[
\kappa=\begin{pmatrix} 
1 & 1 & 0 & 0 & 0 \\
0 & 0 & 0 & 0 & 0 \\
1 & 1 & 1 & 0 & 0 \\
1 & 1 & 0 & 0 & 0
\end{pmatrix}
\in\Mat_{4,5},
\]
$r(\kappa)=(2,0,3,2)$, $r^{+}(\kappa)=(3,2,2,0)$ and
$c(\kappa)=(3,3,1,0,0)=(r^{+}(\kappa))'$.
If $c(\kappa)=\nu'$ then \eqref{Eq_law} is solved by 
$(\la,w)=(\nu,\id)$, leading to our next result.
\begin{proposition}\label{Prop_CT-vnu}
Let $a=(a_1,\dots,a_n)$ be a sequence of positive integers 
and let $\bt$ and $\btau$ be given by 
\eqref{Eq_tset} and \eqref{Eq_tauset} respectively. 
Then, for $v\in\Comp_{n,m}$,
\begin{equation}\label{Eq_CTSchur-2}
D_v(a;\bt)=
D_{v,v^{+}}(a;\bt)=\bigg[\prod_{i=1}^n\prod_{j=1}^{v_i} s_{ij}\bigg]
D_0(a1^m;\btau).
\end{equation}
\end{proposition}

\subsection{Proof of Kadell's conjecture}

Because our results of Section~\ref{Sec_Towards-Kadell} assume that all 
integers of the sequence $a=(a_1,\dots,a_n)$ are positive (it is in 
fact easy to show that all results are true provided at most one of the 
$a_i$ is zero), we need to separately treat the case when some of the
$a_i$ are zero.

\begin{proof}[Proof of \eqref{Eq_Kadell-nul} for positive $a_i$]
Let $v\in\Comp_{n,m}$ be a composition of $m$ such that $\max\{v_i\}<m$.
Now choose $\kappa$ to be a $(0,1)$-matrix in $\Mat_{n,m}$ 
such that 
\begin{equation}\label{Eq_rcnul}
r(\kappa)=(v_1,\dots,v_n)=v \quad\text{and}\quad c(\kappa)=(1^m).
\end{equation}
In other words, each column of $\kappa$ contains a single one and the
row sums of $\kappa$ (none of which has more than $m-1$ ones) form the
composition $v$.
Obviously, there must be a row of $\kappa$ which contains a $0$ followed
by a $1$, since not all ones occur in the same row and no column has
more than a single one.
We also note that the Gale--Ryser theorem does not cause an obstruction
since there are exactly $m!/(v_1!\cdots v_n!)>0$ $(0,1)$-matrices such 
that \eqref{Eq_rcnul} holds.
Finally, Since $c(\kappa)=(1^m)$, \eqref{Eq_law} is solved by 
$(\la,w)=((m),\id)$.
According to Proposition~\ref{Prop_CT-zero} we thus have
\[
D_{v,(m)}(a;\bt)=0.
\]
Equation \eqref{Eq_Kadell-nul} corresponds to the special case 
$t_{ij}=q^{a_j}$.
\end{proof}

\begin{proof}[Proof of \eqref{Eq_Kadell-niet-nul} for positive $a_i$]

Define
\[
W_a^{(k)}(\bt)=
\sum_{\substack{w\in\Symm_n \\ w(n)=k}} t_{R(w)}
\prod_{i=1}^n\frac{1-q^{\A_i}}{1-q^{w(\A_i)}}.
\]
Instead of \eqref{Eq_Kadell-niet-nul} we will derive the more general identity
\begin{equation}\label{Eq_Kadell_t}
D_v(a;\bt)=D_{v,(m)}(a;\bt)=
W_a^{(k)}(\bt)\,
\frac{(q^{\abs{a}})_m}{(q^{\abs{a}-a_k+1})_m}
\prod_{i=1}^n\qbin{\sigma_i-1}{a_i-1} 
\end{equation}
for $v=(0^{k-1},m,0^{n-k})$.
From the following refinement of \eqref{Eq_usum}
\begin{equation}\label{Eq_usum-k}
\sum_{\substack{w\in\Symm_n \\ w(n)=k}} w\bigg(\,\prod_{i=1}^n 
\frac{1-u_i}{1-u_1\cdots u_i}\bigg) \prod_{(i,j)\in R(w)} u_j
=\frac{u_{k+1}\cdots u_n(1-u_k)}{1-u_1\cdots u_n},
\end{equation}
if follows that
\[
W_a^{(k)}(\bt)\big|_{t_{ij}=q^{a_j}}=q^{\sigma_n-\sigma_k}
\frac{1-q^{a_k}}{1-q^{\abs{a}}}
\prod_{i=1}^n \frac{1-q^{\sigma_i}}{1-q^{a_i}}
\]
resulting in \eqref{Eq_Kadell-niet-nul}.

It remains to show \eqref{Eq_Kadell_t} and \eqref{Eq_usum-k}.
Starting with the latter, we will simultaneously prove \eqref{Eq_usum}
(with $m\mapsto n$) and \eqref{Eq_usum-k}. \label{pagelabel}
If we denote the sum sides of these identities by $g(u)$ and $g_k(u)$,
then
\[
g(u)=\sum_{k=1}^n g_k(u)\quad\text{and, for $k\geq 1$,}\quad 
g_k(u)=\frac{u_{k+1}\cdots u_n(1-u_k)}{1-u_1\cdots u_n} \, g(u^{(k)}),
\]
where $u^{(k)}=(u_1,\dots,u_{k-1},u_{k+1},\dots,u_n)$.
These two equations imply a recurrence for $g(u)$. 
Thanks to the initial condition $g(\text{--})=1$ and some telescoping,
this is trivially solved by $g(u)=1$, establishing both
\eqref{Eq_usum} and \eqref{Eq_usum-k}.

To prove \eqref{Eq_Kadell_t} we let $v=(0,\dots,0,m)$ in \eqref{Eq_CTSchur-2}.
Then
\[
D_{(0,\dots,0,m)}(a,\bt)=\big[s_{n1}\dots s_{nm} \big] D_0(a1^m;\btau).
\]
We now essentially repeat the reasoning that led to 
Proposition~\ref{Prop_CT-zero}.
The fact that $t_{ij}=0$ for $i>n$ means that only the $(m+n)!/m!$
permutations of $(1,\dots,n+m)$ contribute to 
$D_0(a1^m;\btau)$ for which each pair of integers in
$n+1,\dots,n+m$ occurs in natural order.
Since we further need to take the coefficient of $s_{n1}\dots s_{nm}$,
only those permutations $w$ contribute for which $t_{R(w)}$ contains the
subword $t_{n,n+1}\dots t_{n,m}$ but none of the letters $t_{i,j+n}$ for
$i<n$ and $j\geq 1$.
These are exactly the $(n-1)!$ permutations of the form
$w=(\pi,n+1,\dots,m+n,n)$, where $\pi\in\Symm_{n-1}$.
For such a permutation, and $a_{n+1}=\cdots=a_{n+m}=1$,
\[
\prod_{i=1}^{m+n} \frac{1-q^{\A_i}}{1-q^{w(\A_i)}}
\qbin{\A_i-1}{a_i-1}=
\frac{(q^{\abs{a}})_m}{(q^{\abs{a}-a_n+1})_m}
\prod_{i=1}^{n-1}\frac{1-q^{\A_i}}{1-q^{\pi(\A_i)}}
\prod_{i=1}^n\qbin{\A_i-1}{a_i-1}.
\]
Thanks to Theorem~\ref{Thm_Poincare-q-Dyson} we thus find
\begin{align*}
D_{(0,\dots,0,m)}(a;\bt)&=
\frac{(q^{\abs{a}})_m}{(q^{\abs{a}-a_n+1})_m}
\prod_{i=1}^n\qbin{\sigma_i-1}{a_i-1} \cdot
\sum_{\pi\in\Symm_{n-1}} t_{R(\pi)}
\prod_{i=1}^{n-1}\frac{1-q^{\A_i}}{1-q^{\pi(\A_i)}}\\
&=W_a^{(n)}(\bt)\,
\frac{(q^{\abs{a}})_m}{(q^{\abs{a}-a_n+1})_m}
\prod_{i=1}^n\qbin{\sigma_i-1}{a_i-1}.
\end{align*}
This settles the $k=n$ case of \eqref{Eq_Kadell_t}.

To deal with the remaining cases, let $s_k$ denotes the $k$th adjacent
(or simple) transposition. 
Then, by the $w=s_k$ case of Lemma~\ref{Lem_sym},
\[
D_{(0^{k-1},m,0^{n-k})}(a;\bt)=t_{k,k+1} 
D_{(0^{k-2},m,0^{n-k+1})}\big(s_k(a);s_k(\bt)\big)\quad 1\leq k\leq n-1.
\]
All we need to do is show that the claimed right-hand side of 
\eqref{Eq_Kadell_t} satisfies this same recursion.
To this end let $w\in\Symm_n$ such that $w(n)=k$ and let $W_a^{(k)}(\bt;w)$
denote the right-hand side but without the sum over $\Symm_n$:
\[
W_a^{(k)}(\bt;w)=
\frac{(q^{\abs{a}})_m}{(q^{\abs{a}-a_k+1})_m}
\prod_{i=1}^n\qbin{\sigma_i-1}{a_i-1} \cdot
t_{R(w)} \prod_{i=1}^n \frac{1-q^{\A_i}}{1-q^{w(\A_i)}}.
\]
It is a straightforward exercise to verify that
\[
W_a^{(k)}\big(\bt;w\big)=
t_{k,k+1} W_{s_k(a)}^{(k+1)}\big(s_k(\bt);\hat{w}\big),
\]
where $\hat{w}$ is the same permutation as $w$ except for the fact that
the numbers $k$ and $k+1$ have swapped position. 
(In particular $\hat{w}(n)=k+1$ as it should.)
Summing $w$ over $\Symm_n$ (such that $w(n)=k$) does the rest.
\end{proof}

\begin{proof}[Proof of Theorem~\ref{Thm_Kadell-conjecture} when not
all $a_i>0$]

We may assume that not all $a_i$ are zero since
\[
D_{v,\la}(0,\dots,0)=\delta_{v,0}\delta_{\la,0}.
\]

Given $a=(a_1,\dots,a_n)$ and $I\subseteq\{1,\dots,n\}$,
let $a^I$ denote the sequence obtained from
$a$ by deleting all $a_i$ for $i\in I$ and $a_I$ the
sequence of deleted $a_i$.
For example, if $n=7$ then $a^{\{3,5,6\}}=(a_1,a_2,a_4,a_7)$
and $a_{\{3,5,6\}}=(a_3,a_5,a_6)$.

Now let $I$ denote the index-set of those $a_i$ that are 
zero, so that the entries of $a^I$ are all strictly positive.
Recalling the third of our remarks made after \eqref{Eq_D_vla} 
we have
\begin{equation}\label{Eq_reduction}
D_{v,(m)}(a)=D_{v^I,(m)}(a^I)\prod_{i\in I} \delta_{v_i,0}.
\end{equation}

To prove that \eqref{Eq_Kadell-niet-nul} for $I\neq\emptyset$ is true 
there are two cases to consider.
First, if $k\in I$ (i.e., $a_k=0$) then the right-hand side vanishes due
to the factor $\delta_{v_k,0}$ in \eqref{Eq_reduction} and the fact
that $v_k=m>0$.
But the left side also vanishes due to the factor $(1-q^{a_k})$.
Next, if $k\not\in I$ then $v_i=0$ for all $i\in I$. 
Hence \eqref{Eq_reduction} simplifies to
\[
D_{v,(m)}(a)=D_{v^I,(m)}(a^I)=
\frac{q^{\A_n-\A_k} (1-q^{a_k})(q^{\abs{a}})_m}
{(1-q^{\abs{a}})(q^{\An-a_k+1})_m} \qbin{\abs{a}}{a},
\]
where the second equality follows from \eqref{Eq_Kadell-niet-nul}
for positive $a_i$.

Finally, to see that \eqref{Eq_Kadell-nul} for $I\neq\emptyset$ is true
we note that $v^{+}\neq (m)$ implies that $\big(v^I\big)^{+}\neq (m)$.
Hence, by \eqref{Eq_Kadell-nul} for positive $a_i$, \eqref{Eq_reduction}
yields
\[
D_{v,(m)}(a)=0. \qedhere
\]
\end{proof}

\subsection{Beyond Kadell}

One can consider more general applications of 
Propositions~\ref{Prop_CT-kappa}--\ref{Prop_CT-vnu} than Kadell's 
conjecture, and in this section we present the full details of one
further example.

In Proposition~\ref{Prop_CT-vnu} take $v=\bar{\la}$ with $\la\in\Part_n$
such that $\la_1=m$. 
Then
\[
D_{\bar{\la}}(a;\bt)
=\bigg[\prod_{i=1}^n\prod_{j=1}^{\la_{n-i+1}} s_{ij}\bigg]
D_0(a1^m;\btau).
\]
As before, we need to determine which permutations in $\Symm_{m+n}$
contribute to the right-hand side.
This can simply be read off from a diagrammatic representation of the
right as we will illustrate through an example. 

Let $m=3$, $n=4$ and $\bar{\la}=(0,1,3,3)$, and represent the set 
$\bs$ as well as the composition $\bar{\la}$ by a filling of an
$m\times n$ rectangle as follows:

\begin{center}
\begin{tikzpicture}[scale=0.6]
\filldraw[blue!20] (6,-4)--(3,-4)--(3,-1)--(4,-1)--(4,-2)--(6,-2)--cycle;
\draw[thick,blue!80] (3,0) rectangle (6,-4);
\draw (3.5,-0.5) node {$t_{15}$};
\draw (4.5,-0.5) node {$t_{16}$};
\draw (5.5,-0.5) node {$t_{17}$};
\draw (3.5,-1.5) node {$t_{25}$};
\draw (4.5,-1.5) node {$t_{26}$};
\draw (5.5,-1.5) node {$t_{27}$};
\draw (3.5,-2.5) node {$t_{35}$};
\draw (4.5,-2.5) node {$t_{36}$};
\draw (5.5,-2.5) node {$t_{37}$};
\draw (3.5,-3.5) node {$t_{45}$};
\draw (4.5,-3.5) node {$t_{46}$};
\draw (5.5,-3.5) node {$t_{47}$};
\end{tikzpicture}
\end{center}
Taking the coefficient of 
$\prod_{i=1}^n\prod_{j=1}^{\la_{n-i+1}} s_{ij}$
means that we need to take the coefficient of 
$$t_{15}^0t_{16}^0t_{17}^0 t_{25}^1t_{26}^0t_{27}^0
t_{35}^1t_{36}^1t_{37}^1 t_{45}^1t_{46}^1t_{47}^1=
t_{25}t_{35}t_{36}t_{37}t_{45}t_{46}t_{47}.
$$
Which permutations in $\Symm_{4+3}$ contribute can now be read off
from the diagram. The numbers $\blue{5},\blue{6},\blue{7}$ need to be
in natural order since $t_{56}=t_{57}=t_{67}=0$.
Because $t_{45},t_{46}$ and $t_{47}$ are covered by the diagram of
$\bar{\la}$ it follows that $4$ occurs after 
$\blue{5},\blue{6},\blue{7}$. 
Because $t_{35},t_{36}$ and $t_{37}$ are covered by $\bar{\la}$ it
follows that $3$ also occurs after $\blue{5},\blue{6},\blue{7}$.
Because $t_{25}$ is covered by $\bar{\la}$, but $t_{26}$ is not, the
number $2$ comes after $\blue{5}$ but before $\blue{6}$.
Because $t_{15}$ is not covered, the number $1$ comes before
$\blue{5},\blue{6},\blue{7}$.
As a result only two permutations contribute:
\[
(1,\blue{5},2,\blue{6},\blue{7},3,4) \quad\text{and}\quad 
(1,\blue{5},2,\blue{6},\blue{7},4,3).
\]

Having worked out this example in full detail it is clear that the
following somewhat simpler diagram than the above staircase encodes
exactly the same information:

\medskip
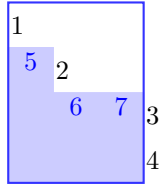
\begin{center}
\begin{tikzpicture}[scale=0.6]
\filldraw[blue!20] (3,0)--(0,0)--(0,3)--(1,3)--(1,2)--(3,2)--cycle;
\draw[thick,blue!80] (0,0) rectangle (3,4);
\draw (0.2,3.5) node {$1$};
\draw (1.2,2.5) node {$2$};
\draw (3.2,1.5) node {$3$};
\draw (3.2,0.5) node {$4$};
\draw (0.5,2.7) node {$\blue{5}$};
\draw (1.5,1.7) node {$\blue{6}$};
\draw (2.5,1.7) node {$\blue{7}$};
\end{tikzpicture}
\end{center}

\noindent
Here numbers occurring in the same column (such as $3,4$) may be 
permuted but numbers in the same row (such as $\blue{6,7}$) have
their relative ordering fixed. 

Given $\la=(\la_1,\dots,\la_n)\in\Comp_n$ such
that $\la_1=m$ let $\mult_i=\mult_i(\la)$ for $0\leq i\leq m$
be the multiplicity of $i$ in $\la$:
\[
\mult_i=\abs{\{\la_j: \la_j=i\}}.
\]
Then the number of numbers in the $i$th column ($0\leq i\leq m$)
of the above-type diagram is given by $\mult_i(\la)$,
so that the sum over $\Symm_{m+n}$ reduces to a sum over
$\Symm_{\mult_0}\times\cdots\times\Symm_{\mult_m}$.
In the case of our example, $\mult_0=1, \mult_1=1, \mult_2=0, \mult_3=2$,
resulting in a sum over $\Symm_2$ instead of $\Symm_7$.

We conclude by applying the above considerations to 
the case of strict partitions, i.e., $\la=(\la_1,\dots,\la_n)$
with $\la_i>\la_{i+1}$ for all $1\leq i\leq n-1$.
Then all multiplicities are $1$ so that there is only one 
remaining permutation:
\begin{multline*}
(\blue{n+1,\dots,n+\la_n},1,\blue{n+\la_n+1,\dots,n+\la_{n-1}},2,\dots \\
\dots,n-1,\blue{n+\la_2+1,\dots,n+\la_1},n).
\end{multline*}
Therefore
\[
D_{\bar{\la}}(a;\bt)
=\prod_{i=1}^n\qbin{\bar{\la}_i+\A_i-1}{a_i-1}.
\]
Applying Lemma~\ref{Lem_sym} with $v=\bar{\la}$ this yields
\[
D_{w^{-1}(\bar{\la})}(a;\bt)
=w\bigg(\,\prod_{i=1}^n\qbin{\bar{\la}_i+\A_i-1}{a_i-1}\bigg)
t_{R(w)},
\]
where
\[
w\bigg(\qbin{\bar{\la}_i+\A_i-1}{a_i-1}\bigg)=
\qbin{\bar{\la}_i+a_{w(1)}+\cdots+a_{w(i)}-1}{a_{w(i)}-1}.
\]
Finally taking $t_{ij}=q^{a_j}$ results in our final proposition.
\begin{proposition}\label{Prop_strict}
Let $\la\in\Part_n$ be a strict partition, i.e., 
$\la_1>\la_2>\cdots>\la_n\geq 0$, and set 
$\bar{\la}=(\la_n,\dots,\la_1)$.
Then, for $a_1,\dots,a_n$ positive integers and $w\in\Symm_n$,
\[
D_{w^{-1}(\bar{\la})}(a)=
w\bigg(\,\prod_{i=1}^n\qbin{\bar{\la}_i+\A_i-1}{a_i-1}\bigg)
q^{\sum_{(i,j)\in R(w)} a_j}.
\]
\end{proposition}
When $w=(n,\dots,2,1)$ is the permutation of maximal length we obtain
\eqref{Eq_strict}.

\section{An application to some $q$-Dyson coefficients}\label{Sec_Questions}

A result concerning coefficients of the $q$-Dyson product
other than the constant term is a theorem of Lv, Xin and Zhou
\cite{LXZ09}, which can be rephrased as an evaluation of
$D_{v,0}(a)$ for certain $v\in\Z^n$ (as opposed to $v\in\Comp_n$) 
as follows.

\begin{theorem}\label{Thm_LXZ}
Let $v\in\Z^n$ such that $\abs{v}=0$, $\max\{v\}\leq 1$ and $v_1=1$.
Let $I$ be the index-set of the positive $v_i$, i.e.,
$I=\{1\leq i\leq n:~v_i=1\}$. Then
\[
\CT\bigg[x^{-v} \prod_{1\leq i<j\leq n} (x_i/x_j)_{a_i} (qx_j/x_i)_{a_j} \bigg]
=\qbin{\An}{a}\sum_{J\subseteq I}(-1)^{\abs{J}} q^{E(J)}
\frac{1-q^{a_J}}{1-q^{1+\abs{a}-a_J}},
\]
where $a_J:=\sum_{j\in J} a_j$ and 
\[
E(J)=\sum_{\substack{1\leq i\leq j\leq n \\[1pt] j\not\in J}} v_i a_j.
\]
\end{theorem}

The technical condition $v_1=1$ can be easily dropped by the application
of a simple transformation which will be described later, but the formula thus
obtained is less attractive. The above theorem extends the equal parameter
case of Stembridge's first layer formulas for characters of 
$\mathrm{SL}(n,\mathbb{C})$ \cite{Stembridge87} 
and contains as special cases a number of earlier conjectures of 
Sills \cite{Sills06}.

In their report the anonymous referee asked the question as to whether 
some of these more general Dyson-type identities follow
from our main theorems. 
In answer to this question we will now show that
Kadell's orthogonality conjecture of Theorem~\ref{Thm_Kadell-conjecture}
indeed implies one of Sills' ex-conjectures \cite[Conjecture 1.2]{Sills06}. 

\begin{theorem}[{Cf. \cite[Corollary 1.4]{LXZ09}}]\label{Thm_Sills-conjecture}
Let $a=(a_1,\dots,a_n)$ be of sequence of nonnegative integers and 
$1\leq r,s\leq n$ a pair of distinct integers. Then
\[
\CT\bigg[(x_r/x_s) \prod_{1\leq i<j\leq n} (x_i/x_j)_{a_i} 
(qx_j/x_i)_{a_j} \bigg]
=-q^{E_{r,s}} \frac{1-q^{a_s}}{1-q^{1+\abs{a}-a_s}}\qbin{\An}{a},
\]
where
\[
E_{r,s}=\chi(r<s)+\sum_{i=s+1}^{r-1}a_i.
\]
\end{theorem}

In the above the summation is understood cyclically, i.e.,
\[
\sum_{i=s+1}^{r-1}a_i=\sum_{i=s+1}^na_i+\sum_{i=1}^{r-1}a_i
\]
when $r<s$. 

In the proof we will exploit some properties of the following transformation.
For a Laurent polynomial $L(x)$, define
the $q$-shifted cyclic action $\gamma$ on $L$ as
\[
\gamma\big(L(x_1,x_2,\dots,x_{n-1},x_n)\big)=L(x_2,x_3,\dots,x_n,x_1/q).
\]
Abbreviate the $q$-Dyson product as $D(a;x)$. Then 
$\CT[L]=\CT[\gamma(L)]$ and 
\[
\gamma^{-1}\big(D(a;x)\big)=D\big(\gamma(a);x\big),
\]
where $\gamma(a):=(a_2,\dots,a_n,a_1)$, cf. \cite[Lemma 2.1]{LXZ09}.

\begin{proof}[Proof of Theorem~\ref{Thm_Sills-conjecture}]
Fixing $1<k\le n$ and $m=1$, \eqref{Eq_Kadell-niet-nul} reads
\[
\CT\bigg[\,\sum_{i=1}^n\frac{1-q^{a_i}}{1-q}
\cdot\frac{x_i}{x_k}\:D(a;x)\bigg]=
\frac{q^{a_{k+1}+\dots+a_n}(1-q^{a_k})}{1-q^{\abs{a}-a_k+1}}\qbin{\An}{a}.
\] 
Applying the same formula but with $k-1$ instead of $k$ and 
$\gamma(a)$ instead of $a$, and with the convention $a_{n+1}=a_1$, we obtain
\[
\CT\bigg[\,\sum_{i=1}^n\frac{1-q^{a_i+1}}{1-q}
\cdot \frac{x_i}{x_k}\:D(\gamma(a);x)\bigg]=
\frac{q^{a_{k+1}+\dots+a_{n+1}}(1-q^{a_k})}{1-q^{\abs{a}-a_k+1}}\qbin{\An}{a}.
\] 
By the above-mentioned properties of $\gamma$, the left hand side
may be reinterpreted as
\[  
\CT\bigg[\bigg(\frac{1-q^{a_1}}{q(1-q)}
\cdot \frac{x_1}{x_k}
+\sum_{i=2}^n\frac{1-q^{a_i}}{1-q}
\cdot \frac{x_i}{x_k}\bigg)D(a;x)\bigg].
\]
Subtracting the two equations thus obtained we find that
\[
\big(1-q^{-1}\big) 
\CT\bigg[\frac{1-q^{a_1}}{1-q}
\cdot\frac{x_1}{x_k}\,D(a;x)\bigg]=
\frac{(1-q^{a_1})q^{a_{k+1}+\dots+a_n}(1-q^{a_k})}{1-q^{\abs{a}-a_k+1}}
\qbin{\An}{a}.
\]
This establishes Theorem \ref{Thm_Sills-conjecture} in the special case
when $r=1$. Using the convention $x_0=x_n$,
the full content of the theorem follows by a repeated 
application of the identity
\[
\CT\big[(x_r/x_s)D(a;x)\big]=q^{\chi(r=1)-\chi(s=1)}
\CT\big[(x_{r-1}/x_{s-1})D(\gamma(a);x)\big]. \qedhere
\]
\end{proof}

As briefly outlined below, Theorem~\ref{Thm_Sills-conjecture} 
can also be obtained 
directly using the polynomial method we employed to prove 
Theorem~\ref{Thm_Poincare-q-Dyson}.
Note that we may clearly assume that $s=1$ and that all $a_i$, 
with the possible exception of $a_r$, are positive. 
Accordingly, we want to compute the coefficient of the monomial
\[
(x_1/x_r) \prod_{i=1}^n x_i^{\abs{a}-a_i}
\]
in the homogeneous polynomial
\[
F(x)=\prod_{1\le i<j\le n}
\Bigg(\prod_{k=0}^{a_i-1}{\big(x_j-x_iq^k\big)}\cdot
\prod_{k=1}^{a_j}{\big(x_i-x_jq^k\big)} \Bigg).
\]
This can be done efficiently using Lemma~\ref{Lem_interpol} if the sets $A_i$
therein are chosen as follows. Let $A_i=\{q^{\alpha_i}:~\alpha_i\in B_i\}$,
where
\[
B_1=\{0,1,\ldots,\abs{a}-a_1+1\},\quad
B_r=\{0,\ldots,\abs{a}-a_r\}\setminus \bigg\{\sum_{i=2}^{r-1}a_i\bigg\}
\]
and $B_i=\{0,\ldots,\abs{a}-a_i\}$ otherwise. The sets $A_i$ clearly
have the right cardinalities. 
Now there is exactly one element $\bc\in A_1\times\dots\times A_n$
such that $F(\bc)\ne 0$. Indeed, 
assume that $c_i=q^{\alpha_i}\in A_i$ and $F(\bc)\ne 0$.
Then all $\alpha_i$ are distinct, with the possible exception of
$a_r$ being equal to $a_j$ for some $j>r$. Moreover, 
$\alpha_j\ge \alpha_i+a_i+\chi(j<i)$
holds for $\alpha_j> \alpha_i$. Next consider the unique permutation 
$\pi\in \Symm_n$ for which 
\[
0\le \alpha_{\pi(1)}\le \alpha_{\pi(2)}\le\dots\le \alpha_{\pi(n)}\le 
\abs{a}-a_{\pi(n)}+\chi\big(\pi(n)=1\big).
\]
Here $\pi(i)=r$ is assumed in case of $\alpha_{\pi(i+1)}=\alpha_{\pi(i)}$.
We can argue that
\[
\abs{a}-a_{\pi(n)}=
\sum_{i=1}^{n-1}a_{\pi(i)}\le 
\sum_{i=1}^{n-1}\big(\alpha_{\pi(i+1)}-\alpha_{\pi(i)}\big)=
\alpha_{\pi(n)}-\alpha_{\pi(1)}\le \abs{a}-a_{\pi(n)}+1.
\]
Notice that the first inequality is strict if $\pi$ is not the identity
permutation, while the second inequality is strict if $\pi(n)\ne 1$.
Suppose that $\pi(n)=1$. Then there is exactly one $i$ such that
$\pi(i+1)<\pi(i)$. This implies that $\pi=(2,3,\dots,n,1)$ and
$\alpha_r=a_2+\dots+a_{r-1}$, which is not possible in view of 
the choice of $B_r$. Thus $\pi(n)\ne 1$, implying $\pi= \id$ and 
$\alpha_i=a_1+\dots +a_{i-1}$ for every $i$. 
Substituting these values into
\[
\frac{F(c_1,c_2,\dots,c_n)}{\phi_1'(c_1)\phi_2'(c_2)\cdots\phi_n'(c_n)}
\]
one recovers the case $s=1$ of Theorem \ref{Thm_Sills-conjecture}
without any difficulty.
\bigskip

As a final remark we mention that Zhou \cite{Zhou11} successfully applied
Theorem \ref{Thm_LXZ} to prove another conjecture of Kadell
\cite[Conjecture 2]{Kadell98} related to the Dyson product. 
The referee also asked if our results could be used to obtain 
Zhou's theorem. At present we do not know how to do this.
However, the polynomial method can be used to prove
a special case of Kadell's \cite[Conjecture 3]{Kadell98}, see
\cite{Karolyi13}. 

\subsection*{Acknowledgement}
We thank David Bressoud for helpful correspondence and the anonymous
referee for raising some interesting questions, 
leading to the material presented
in Section~\ref{Sec_Questions}.

\bibliographystyle{amsplain}

\end{document}